\documentclass[11pt]{amsart}
\usepackage{amsmath,mathtools}
\usepackage{amsmath, amsthm, amssymb, amsfonts, enumerate}

\usepackage[colorlinks=true,linkcolor=blue,urlcolor=blue]{hyperref}
\usepackage{dsfont}
\usepackage{color}
\usepackage{geometry}
\usepackage{todonotes}
\usepackage{epstopdf}
\usepackage{changes}
\usepackage{amsmath,mathtools}
\usepackage{changes}
\usepackage{relsize}
\newcommand{\stkout}[1]{\ifmmode  \text{\sout{\ensuremath{#1}}}\else\sout{#1}\fi}

\newtheorem{theorem}{Theorem}[section]
\newtheorem{remark}[theorem]{Remark}
\newtheorem{assumption}[theorem]{Assumption}
\newtheorem{lemma}[theorem]{Lemma}
\newtheorem{proposition}[theorem]{Proposition}

\newtheorem{definition}[theorem]{Definition}

\def \R{\mathbb{R}}

\definecolor{red}{rgb}{1.0,0.0,0.0}
\def\blu#1{{\textcolor{blu}{#1}}}
\definecolor{blu}{rgb}{0.0,0.0,1.0}
\def\blu#1{{\textcolor{blu}{#1}}}
\definecolor{gre}{rgb}{0.03,0.50,0.03}

\definecolor{darkviolet}{rgb}{0.58, 0.0, 0.83}

\definecolor{blu}{rgb}{0, 0, 0} 

\def \eps{\varepsilon}

\title[Linear-Quadratic Mean Field Games in Hilbert spaces]{Linear-Quadratic Mean Field Games in Hilbert spaces}
\author[Federico]{Salvatore Federico}
\author[Ghilli]{Daria Ghilli}
\author[Gozzi]{Fausto Gozzi}

\address{S.~Federico: Dipartimento di Matematica, University of Bologna, Porta San Donato, Bologna, Italy}
\email{\href{mailto:s.federico@unibo.it}{s.federico@unibo.it}}
\address{D.~Ghilli: Dipartimento di Scienze Economiche e Aziendali, University of Pavia, Via San Felice al Monastero 5, Pavia, Italy}
\email{\href{mailto: daria.ghilli@unipv.it}{daria.ghilli@unipv.it}}
\address{F.~Gozzi: Dipartimento di Economia e Finanza, LUISS University of Rome, Vialee Romania 32, Roma, Italy.}
\email{\href{mailto: fgozzi@luiss.it}{f.gozzi@luiss.it}}

\date{\today}

\numberwithin{equation}{section}

\begin{document}

\begin{abstract}
{This paper represents  the first attempt to develop a theory for linear-quadratic mean field games in possibly infinite dimensional Hilbert spaces. As a starting point, we study the case, considered in most finite dimensional contributions on the topic, where the dependence on the distribution enters just in the objective functional through the mean.
This feature allows, similarly to the finite-dimensional case, to reduce the usual mean field game system to a Riccati equation and a forward-backward coupled system of abstract evolution equations. Such a system is completely new in infinite dimension and no results have been proved on it so far. We show the existence and uniqueness of solutions for such system, applying a delicate approximation procedure. We apply the results to a production output planning problem with delay in the control variable.}
\end{abstract}

\maketitle

\smallskip

{\textbf{Keywords}}:  Mean field games, Infinite dimensional linear-quadratic control, delay equations.

\smallskip

{\textbf{MSC2010 subject classification}}:  49L20, 70H20, 93E20,  47D03

 \tableofcontents

\section{Introduction}\label{sec:intro}

The theory of Mean Field Games (MFGs, hereafter, for short)
is a powerful tool to study situations where many forward-looking
players interact through the distributions of their state/control variables.
The starting foundation of this theory is usually dated in 2006, with the seminal papers by Lasry-Lions on the one side and by Huang-Caines-Malhamé \cite{LL1, LL2, LL3, L, HCM} on the other side. Since then, a huge amount of work has been done in this area, both from the theoretical and the applied viewpoint; far to be exhaustive, we quote, as benchmark references for our scopes, the nowadays classical contributions \cite{BeFrYa, CarmonaDelarueBook, CarmonaDelarue2, CDLLbook}.

However, an interesting topic in this area is still largely missing; that is,
the case when the state space of the system, and possibly also of
the control space, is not finite-dimensional. In  control theory, this kind of problem arises, e.g., when the dynamics of the agent
depends on other variables beyond time, such as age or space, or when such dynamics is path-dependent.\footnote{In these cases, a typical approach consists in lifting the dynamics into an infinite dimensional Ordinary Differential Equation (ODE for short) containing unbounded terms.
Roughly speaking, the infinite dimensionality encloses the dependence on other variables (space, age) and/or the dependence on the past paths.}

This paper constitutes a first attempt to fill this gap.
We do it dealing with infinite dimensional MFGs in the Linear-Quadratic (LQ) case; \blu{in another paper (see \cite{FeGoSw}),  the more general nonlinear cases under some assumptions (global Lipschitz regularity of the resulting Hamiltonian) that are not satisfied in the LQ case.}

More precisely, we focus (as most of the literature on LQ MFGs in finite dimension, see e.g. \cite{BeFrYa,BeSuYa}) on MFGs where the dynamics of the representative player is linear and independent on the distribution;  the coupling  enters only in the cost functional (which is purely quadratic) through the mean of the distribution of the players. This kind of structure is still suitable to investigate a range of problems arising in several applications, see Section 5 below.

\subsection{Some literature}

First of all we recall that, beyond the basic references on MFGs in finite dimension recalled at the beginning of the introduction,
various papers have studied the LQ MFG case
with dependence on the distribution entering just through the mean in the objective functional.
We recall, in particular one of the papers establishing the first steps of the theory \cite{CH}, where  a set of decentralized control laws for the individuals is obtained, and the $\eps$-Nash equilibrium property is proved for such set (see also \cite{LiZh}); moreover, there one finds also some examples to specific situations, such as the production output planning problem, a suitable extension of which is the object of our Section 5.
Other references, whose finite dimensional techniques have been one of the departure points of our work, are the book \cite{BeFrYa} (Chapter 6, Sections $6.1$-$6.3$) and the paper \cite{BeSuYa}. We also refer to the books \cite{CarmonaDelarueBook, CarmonaDelarue2} (and the references therein) which provide a complete study of the probabilistic approaches to MFG and, in their Chapters 2 and 7, investigate some classes of LQ problems.
We also mention \cite{CaFoSu} for an application to systemic risk. Finally we mention \cite{Bardi12, BaPri} the first one giving explicit solutions to a class of LQ problems in one dimension, where the objective function minimized by the players is computed as an ergodic average over an infinite horizon and the second in dimension greater than one giving necessary
and sufficient conditions for the existence and uniqueness of quadratic-Gaussian solutions
in terms of the solvability of suitable algebraic Riccati and Sylvester equations.
\textcolor{blu}{We also mention that recently a new stream of research has been focusing on the so-called \textit{submodular Mean Field Games}, see e.g. \cite{DiFe, DiFeFiNe, DiFeFiNe2}. The submodularity condition allows  to prove the existence of MFG solutions without using a weak formulation or the notion of relaxed controls and using, instead, probabilistic arguments and a lattice-theoretical approach.}


We now pass to references for the infinite dimensional setting used in this paper.
A general treatment of infinite dimensional ODEs, viewed as abstract evolution equations, is well established since many decades, especially in the deterministic case; here we just quote two standard references (cf. \cite{DPB, DaPraZa}) one for the deterministic and one for the stochastic case, which is the one of our interest.

Relying on that, starting from the work of Barbu and Da Prato \cite{BDPbook} a large amount of work has been done in the last 40 years on stochastic optimal control and Hamilton-Jacobi-Bellman (HJB) equations in Hilbert spaces and is nowadays quite well-established  too: we may mention \cite{FGS} for an extended overview of this theory, including results and references.

On the other hand, also the theory of Fokker-Planck-Kolmogorov (FPK) equations in infinite dimensional spaces has attracted the attention in the last decade with some valuable contributions: we may mention \cite{ BoDaPraRoSh, BoDaPraRoSh2,  BoDaPraRo3, BoDaPraRo5, DaPra, DaPraFlaRo}. Since MFGs analytically consist in coupled forward-backward systems of HJB and FPK equations, it is therefore natural to try to  merge the aforementioned (separate) theories to fill the gap of a missing theory of MFGs in infinite dimensional spaces.

The only papers about MFG in infinite dimensional spaces we are aware of is represented by
\cite{FZ}, where a specific example of an LQ case is treated in a setting which is different
than the one of the present paper, and \blu{ the more recent paper \cite{liu2024hilbert}, where MFGs of LQ with a structure more general than ours is address, but only for short time horizon.}

\subsection{A Sketch of our setting}
\textcolor{blu}{
We begin by providing an overview of the general framework of Mean Field Games (MFG) before introducing our specific setting.  
The study of equilibria in $N$-player games when $N$ is large is a topic of significant interest in many applications, but it poses considerable challenges. In particular, analyzing closed-loop Nash equilibria involves solving a complex system of $N$ coupled Hamilton-Jacobi-Bellman (HJB) equations.  
Mean Field Games (MFG) theory formally arises by taking the limit of an $N$-player game as $N \to \infty$, with the hope that the resulting limit system is more tractable and that its properties provide insights into the equilibria of the $N$-player game for sufficiently large $N$.  
This  research program can be  carried out under three main assumptions (see \cite{CarmonaDelarueBook}):  
(i) the players are symmetric;  
(ii) the interaction is of \textit{mean field type}, meaning that only the overall distribution influences the agents' decisions, while the specific actions of individual players have no direct effect;  
(iii) the players are ``small,'' in the sense that the decisions of a single player do not impact the mean-field system.  
Under these assumptions, it is possible to study the limit problem (at least in some meaningful cases) and establish connections between the limit system and the original $N$-player game.  
More precisely, it can be shown that any solution of a Mean Field Game corresponds to an $\varepsilon$-Nash equilibrium of the associated $N$-player game (see \cite{CarmonaDelarueBook}, Part II, Chapter 6, Section $6.1$). Furthermore, in certain cases (though requiring much greater technical effort), one can prove the convergence of Nash equilibria from the $N$-player game to the solution of the corresponding Mean Field Game (see again  \cite{CarmonaDelarue2} (Part II, Chapter 6, Sections $6.2$ and $6.3$) and \cite{CDLLbook} (Chapter 8, Section $8.2$).
}


We now provide a sketch of our setting and of the
resulting MFG.
Let $H$ be a separable \textcolor{blu}{Hilbert space} and $\mathcal{P}(H)$ the space of probability measures in $H$.
Let us consider the following stochastic optimal control problem with finite horizon $T>0$. The controlled dynamics of a representative agent starting at time $t\in[0,T)$ and
dealing in a large population of agents, evolves in a separable Hilbert space  $H$ and  according to a linear controlled SDE of the form
\begin{equation}\label{eqn:stateeq}
dX(s)=[AX(s)+B\alpha(s)]ds+\sigma dW(s), \ \ \ X(t)=x,
\end{equation}
where
\begin{enumerate}[(i)]
\item $W$ is a cylindrical Wiener process defined on a filtered probability space and valued in another separable Hilbert space $K$, and $\sigma\in\mathcal{L}(K;H)$ is a suitable diffusion coefficient \textcolor{blu}{and $\mathcal{L}(K;H)$ is the Banach space of bounded linear operators from $K$ to $H$ endowed with the usual sup-norm (see Section \ref{sec:notset} for further details)};
\item $A,B$  are suitable linear operators and  $\alpha(\cdot)$ is the control process taking values in some control space $U$ and lying in a set of admissible processes $\mathcal{A}$.
\end{enumerate}

The aim of this representative agent
 is to minimize a cost functional also depending on the overall distribution of the states of the other agents $m: [0,T]\to \mathcal{P}(H)$, such as
$$
J(t,x,\alpha)=
\mathbb{E}\left[\int_t^T f(X(s), m(s),\alpha(s) )ds+h(X(T),m(T))\right],
$$
where $f:H\times \mathcal{P}(H)\times U\to\R$ and  $\textcolor{blu}{h}:H\times \mathcal{P}(H)\to\R$ are given measurable functions.
The value function of the above control problem is
$$
V(t,x)=\inf_{\alpha(\cdot)\in\mathcal{A}}J(t,x;\alpha).
$$
The HJB equation associated to $V$ is the following infinite dimensional parabolic PDE:
\begin{equation}\label{eq:HJBintro}
{\partial}_{t}v(t,x)+\frac{1}{2} \mbox{Tr}\big[ \sigma\sigma^{*}D^2v(t,x)\big]+\langle Ax,Dv(t,x)\rangle_{H}+\mathcal{H}_{min}(x,m(t),Dv(t,x)))=0,
\end{equation}
with terminal condition $v(T,x)=h(x,m(T))$, where  $\langle\cdot,\cdot\rangle_{H}$ denotes the inner product in $H$;  $D,D^{2}$ denote \textcolor{blu}{the first and second order Fréchet derivatives} with respect to the $x$ variable;
and where
\begin{equation}\label{HJB}
\mathcal{H}_{min}(x,m,p):=
\textcolor{blu}{\inf}_{\alpha\in U} \big\{f(x,m,\alpha) + \langle B\alpha,p\rangle_{H}\big\}.
\end{equation}
Call $G(x,m,p)$ the argmax of the above formula  and assume that it is a unique point.
Heuristically speaking, the HJB equation \eqref{eq:HJBintro} allows, given the path of $m(\cdot)$, to find the optimal feedback strategy $$\alpha^*(t)=G(X(t),m(t),Dv(t,X(t)))$$ of the representative agent in terms of $Dv$.
As well known (see e.g. \cite{CarmonaDelarueBook} and \cite{CDLLbook}), denoting by $X^{*}$ the optimally controlled state --- depending on the given $m(\cdot)$ --- and imposing the consistency condition $\mathcal{L}(X^{*}(s))=m(s)$ for every $s$, \textcolor{blu}{where $\mathcal{L}(X)$ denotes the law of the random variable $X$}, one sets a problem that can be interpreted as the limit, as the number $N$ of agents tends to $\infty$, of the Nash equilibrium of the symmetric non-cooperative $N$-players game in which the strategic interactions among agents only depend on the evolution of the probability distribution $m(\cdot)$ of state variables of the agents. The above consistency condition rewrites as a FPK equation for the distribution $m(\cdot)$, which is formally written as:
\begin{equation}\label{FP1}
\partial_{t}m\textcolor{blu}{(t,dx)}- \frac{1}{2}\mbox{Tr} [\sigma \sigma^{*}D^2m\textcolor{blu}{(t,dx)}]+\mbox{div}\left(D_p\mathcal{H}_{min}(x,m\textcolor{blu}{(t,dx)},Dv\textcolor{blu}{(t,x)})
\right)=0,
\end{equation}
with initial condition $m(0,dx)=m_0(dx)$, where $m_0$ is the initial distribution of  the agents' population \textcolor{blu}{and $D_p$ denotes the gradient of $\mathcal{H}$ with respect to  the last variable}.

Since we assume that $f$ and $\textcolor{blue}{h}$ are purely quadratic and depends just on the mean of $m(\cdot)$ (see \eqref{f}-\eqref{h}) the MFG system \eqref{eq:HJBintro}-\eqref{FP1} can be reduced to a system of abstract ODEs. Indeed,  arguing as in the finite dimensional case (see \cite[Ch.\,6]{BeFrYa} and \cite{BeSuYa}),
one defines the variable
$
z(t)=\int_H\xi m(t, d\xi),
$
and one guesses a quadratic structure
$v(t,x)=\frac{1}{2}\langle P(t)x,x\rangle+\langle r(t),x\rangle+s(t)$ for the solution of the HJB equation, with unknown $P(\cdot),r(\cdot),s(\cdot)$.
With this guess, the HJB-FPK system \eqref{f}-\eqref{h}
is rephrased in the following backward Riccati equation for $P(\cdot)$ (cf. \eqref{Riccati:P}):
\begin{equation}\label{Riccati:Pintro}
\begin{cases}
\displaystyle{P'(t)+P(t)A+A^*P(t)-P(t)BR^{-1}B^*P(t)+Q+\overline Q=0,}\\\\
P(T)=Q_T+\overline Q_T,
\end{cases}
\end{equation}
and in the coupled system in $(r(\cdot), z(\cdot))$ (cf. \eqref{eq:z} and \eqref{eq:r}):
\begin{equation}\label{eq:zintro}
\begin{cases}
\displaystyle{z'(t)=(A-BR^{-1}B^*P(t))z(t)-BR^{-1}B^*r(t),}\\\\
z(0)=z_0 :=\displaystyle{\int_{H} x\,\,m_0(dx)} \in H.
\end{cases}
\end{equation}
\begin{equation}\label{eq:rintro}
\begin{cases}
\displaystyle{r'(t)+ (A^*-P(t)BR^{-1}B^*)r(t)-\overline Q Sz(t)=0},\\\\
r(T)=-\overline Q_TS_Tz(T).
\end{cases}
\end{equation}
While the Riccati equation
\eqref{Riccati:Pintro} falls into the already established literature on infinite dimensional LQ control (see e.g. \cite{DPB})\textcolor{blu}{,}
the latter coupled forward-backward system
\eqref{eq:zintro} and \eqref{eq:rintro}
is completely new and need to be studied from scratch.

\subsection{Our results and methods}

Existence and uniqueness of solutions for the Riccati equation \eqref{Riccati:Pintro} involving $P(\cdot)$ is, as said above, well known in the literature (see e.g. \cite{CP74}).
Proposition \ref{lem:boundriccati} recalls such basic result, together with some estimates which turn out to be useful in solving the forward-backward system for $(r(\cdot),z(\cdot))$. To establish the existence for the
forward-backward system
\eqref{eq:zintro}-\eqref{eq:rintro}
the idea is to decouple such system looking, similarly to the finite dimensional case, for solutions in the form
\begin{equation}\label{eqn:zetaetaintro}
r(t)=z(t)\eta(t)
\end{equation}
for some $\eta(\cdot)  :  [0,T] \to \Sigma(H)$, where $\Sigma(H)$ denotes the set of linear bounded self-adjoint operators from $H$ to $H$.

We have to mention that the way of performing this decoupling here is much more delicate with respect to the finite dimensional case.
The main difficulty, as typically happens in dealing with infinite dimensional dynamics, is represented by the fact that the operator $A$ is unbounded, hence we cannot rely on the notion of classical (i.e., $C^{1}$) solutions.
The way to overcome such difficulty, in the infinite dimensional literature, is to employ
weaker concepts of solutions\footnote{Here we use
the concept of \emph{mild solution}, based on a generalization of the finite dimensional \emph{variation of constants formula}.}
and to develop suitable approximations procedures. While this procedure has been already worked out to study Riccati equation like \eqref{Riccati:Pintro}, this is not the case for
the case of our forward-backward system
\eqref{eq:zintro}-\eqref{eq:rintro}
which presents different structure and difficulties.
We outline the path we follow.
First we look at the system
\eqref{eq:zintro}-\eqref{eq:rintro}
where we substitute $A$ with its Yosida's approximants $(A_{n})_{n\in \mathbb{N}}$.
Then use the decoupling idea of \eqref{eqn:zetaetaintro} to get a Riccati equation for an approximating object $\eta_{n}$ and, in turn, decoupled ODEs for approximating objects $z_{n}$ and $r_{n}$. Then, dealing with the weak topology of $H$, we take the limit as $n\to \infty$, relying on Ascoli-Arzel\'a's Theorem in infinite dimension, to get the (mild) solution to the original forward-backward system of ODEs.
This result is the first main result of the paper proven in Theorem \ref{Th:existence}.

As underlined above, the existence for the forward-backward system in $z_n, r_n$ relies on the existence of a solution $\eta_n$ for the associated Riccati equation (see assumption $3.2$ of Theorem \ref{Th:existence}). In
the appendix we prove existence results for such Riccati equation in two cases:  the first one, more standard, for small time horizon (cf. Proposition \ref{prop:tsmall});
 the second one, under the further assumption  of nonnegativity of the operators $-\overline QS$ (cf. Proposition \ref{prop:positive}).

As for the uniqueness issue,
again we analyze two cases: the first one for small time horizon (cf. Proposition \ref{prop:small});
 the second one, under the further assumptions  i) and ii) of Theorem \ref{Th:uniqueness} on the operators $-\overline QS$ and $-\overline Q_TS_T$. This result is proven in Theorem \ref{Th:uniqueness} which is the second main result of the paper.
 Roughly speaking, these assumptions are satisfied either if the two above operators are positive definite or if they are a projection on a closed subspace of $H$.
Note that the conditions on these operators are in the spirit of the classical monotonicity conditions of Lasry-Lions ensuring uniqueness (see \cite{LL3}, Theorem $2.4$), corresponding here to the case when
 $-\overline QS$ and $-\overline Q_TS_T$ are positive definite. In our context, this allows to treat problems where the agent is willing to place themself in the opposite position with respect to the mean of the overall system (see \eqref{f} and \eqref{h})\footnote{In \cite{BeSuYa},  where the finite dimensional case is studied, the authors prove uniqueness by requiring that the time horizon is sufficiently small with respect to the certain operators involving the data of the  problem (see \cite[Prop. 3.5, \,3.6]{BeSuYa}).}.

\subsection{On the applications}

In Section \ref{sec:appl}, we propose an application to a production output planning example with delay.  In this example, the firms supply the same product to the market and the production adjustments are affected also by the the past history of investments (so called \emph{time-to-build}). The aim of the firm is to find a production level which is close to the price. We are able to apply our results to this example.

We believe that our techniques can be adapted to  more general cases. In particular we mention two of them.
First the case in which there is an additional linear dependence on the distribution and on the state in the objective functional (that is, in \eqref{f} and \eqref{h}).
Second, the case like the one of
\cite{FZ} where the operators $-\bar QS, -\bar Q_TS_T$ have the opposite sign with respect to ours.\footnote{Roughly speaking, this means that in our case the agents  pay an additional cost if they stay near to the mean (see assumption ii) and iii) of Theorem \ref{Th:uniqueness}) whereas in \cite{FZ} the agent pay an extra cost if they stay far from the mean.} Indeed in both cases there seems to be room for improvement of our results that will be the subject of forthcoming research.

\subsection{Plan of the paper}
Section 2 is devoted to write write down  setup of our problem together with a formal derivation
of the system \eqref{Riccati:Pintro}-\eqref{eq:rintro}-\eqref{eq:zintro}
from the HJB-FPK system \eqref{eq:HJBintro}-\eqref{FP1},
and with the rigorous definition of solution of our LQ MFG (Definitions \ref{df:solsyst} and \ref{df:solMFG}).

Section 3 is completely to the existence result, Theorem \ref{Th:existence}. Since the existence result strictly depends on the somehow implicit Assumption 3.2, we present, in Appendix A, two results that provide reasonably checkable conditions that guarantee that Assumption 3.2 is satisfied.
Section 4 is  devoted to the uniqueness results, Theorem \ref{Th:uniqueness}.
Section 5 is devoted to illthe ustrate an example.

\section{Notations and formal setting}\label{sec:notset}
Let $(H,\langle\cdot,\cdot\rangle_{H}),  \, (K,\langle\cdot,\cdot\rangle_{K}),\,  (U,\langle\cdot,\cdot\rangle_{U})$
be separable real Hilbert spaces and  denote by $|\cdot|_{H}$, $|\cdot|_{K}$, $|\cdot|_{U}$, respectively, the norms induced by corresponding inner products; unless differently specified, on $H,K,U$ we consider the strong topologies, i.e. the ones induced by their norm.
We denote by $\mathcal{L}(H)$ the space of linear bounded operators $L:H\to H$, by
$\Sigma(H)$ the subspace of $\mathcal{L}(H)$ of self-adjoint operators, by
$
\Sigma^+(H)
$
the space of self-adjoint nonnegative operators and by $\Sigma^{++}(H)$ the space of self-adjoint positive operators.
The space $\mathcal{L}(H)$ is considered endowed with usual sup-norm $$\|L\|_{\mathcal{L}(H)}:=\sup_{|x|_H=1}{|Lx|_{H}},$$ which makes it a Banach space. With respect to this norm,   $\Sigma(H)$ and $\Sigma^{+}(H)$ are closed. Similar notations are used with respect to the spaces $K,U$.

Similarly, we consider the Banach space $\mathcal{L}(U;H)$ of bounded linear operators from $U$ to $H$ endowed with the usual sup-norm
$$\|L\|_{\mathcal{L}(U;H)}:=\sup_{|x|_U=1}{|Lx|_{H}}.$$
By  $\mathcal{P}(H)$ we denote the space of regular probability measures on $H$, \textcolor{blu}{where (see Def.\,1.1 in  \cite{Parthasarathy67}) a regular probability measure $\mu$ on $H$ is a probability measure defined on $\mathcal{B}(H)$ such that, for all Borel sets $B$, one has $\sup_F \mu(F) = \mu(B) = \inf_G\mu(G)$, where $F$ ranges over all closed sets contained in $B$, and $G$ ranges over all open sets containing $B$. The space $\mathcal{P}(H)$  is endowed with the Borel $\sigma$-algebra $\mathcal{B}(\mathcal{P}(H))$ induced by the norm of the total variation, that is, the Borel $\sigma$-algebra generated by the open sets of the topology induced by the norm of  the total variation.\footnote{\blu{We recall the notion of norm of the total variation (see Def.\,3.1.4 in  \cite{Boga}). Consider a signed measure $\mu$ on $H$. It is possible to defined two set functions $\overline W (\mu, \cdot)$ and $\underline{W}(\mu, \cdot)$ respectively called \textit{upper variation} and \textit{lower variation}, as follows
$$
\overline{W}(\mu, E)=\sup\{\mu(A)\, |\, A \in \mathcal{B}(H), A\subset E\} \quad \forall E \in \mathcal{B}(H)
$$
$$
\underline{W}(\mu, E)=\sup\{-\mu(A)\, |\, A \in \mathcal{B}(H), A\subset E\} \quad \forall E \in \mathcal{B}(H).
$$
The variation of $\mu$ is the set function
$$
|\mu|(E)=\overline{W}(\mu, E)+\underline{W}(\mu, E) \quad \forall E \in \mathcal{B}(H)
$$
and its \textit{total variation} is defined as the value of this measure on the whole space of definition, i.e.
$$
\|\mu\|=|\mu|(H).
$$}
}}
For a possibly unbounded linear operator $A\, : H\to H$, we denote by $D(A)$ its domain.

\medskip

Next, we consider the  objects defined in the following assumption, which will be standing throughout the paper.
\begin{assumption}\label{assumptions}
\begin{enumerate}
[(i)]
\item[]
\item $A: D(A)\subset H\to H$ is a closed densely defined linear operator generating  a $C_{0}-$semigroup  $(e^{tA})_{t\geq 0}$ on $H$;
 \smallskip
 \item  $ B \in \mathcal{L}(U; H)$;
 \smallskip
\item    $Q_T, Q \in \Sigma^+(H)$;
 \smallskip

 \item $\overline Q_T, \overline Q \in {\Sigma}(H)$;
 \smallskip
\item $Q+\overline Q \in \Sigma^+(H), Q_T+\overline Q_T \in \Sigma^+(H)$;

  \item $ R \in \Sigma(U)$ such that  $\langle R\alpha,\alpha\rangle_{U}\geq \varepsilon |\alpha|_U^{2}$ {for some} $\varepsilon>0$\footnote{Note that, under these assumptions,  $R$ is invertible and  $R^{-1}\in \Sigma^{+}(U)$.};
   \smallskip

  \item $S,S_{T}\in \mathcal{L}(H;H)$;
   \smallskip

  \item $\sigma\in \mathcal{L}(K;H)$.
 \end{enumerate}
\end{assumption}

\medskip

Let $T>0$ denote a time horizon. Given $x\in H$, $\mu\in \mathcal{P}(H)$, $\alpha\in U$, set
\begin{equation}\label{f}
f(x, \mu,\alpha):=\frac{1}{2}\left[\langle R \alpha, \alpha\rangle_{U}+ \langle Qx, x\rangle_H+\left\langle \overline Q\left(x-S\int_H\xi\, \mu(d\xi)\right), \  x-S\int_H\xi\, \mu(d\xi)\right\rangle_H\right],
\end{equation}
and
\begin{equation}\label{h}
h(x,\mu):=\frac{1}{2}\left[\langle Q_Tx, x\rangle_H +\left\langle \overline Q_T\left(x-S_T\int_H\xi\, \mu(d\xi)\right), \ x-S_T\int_H \xi\, \mu(d\xi)\right\rangle_H\right].
\end{equation}
We are interested in the solvability of the  forward-backward coupled system of PDEs \eqref{eq:HJBintro}-\eqref{FP1}, when $f,h$ are as above; more explicitly,
\begin{equation}\label{HJB:bis}
\begin{small}
\begin{cases}
\displaystyle{-v_{t}(t,x)=\frac{1}{2} \mbox{Tr}[ \sigma\sigma^{*}D^2\textcolor{blu}{v}(t,x)]-\frac{1}{2}\langle B R^{-1}B^*Dv(t,x), Dv(t,x)\rangle_H+\langle Dv(t,x), Ax\rangle_H}\\\\
\qquad \qquad \ \ \ +\displaystyle{
\frac{1}{2}\left[\langle Qx, x\rangle_H +\left\langle \overline Q\left (x-S\int_H \xi m(t,d\xi)\right),\  x-S\int_H\xi m(t,d\xi)\right\rangle_H\right],}\\\\\\
v(T,x)=\displaystyle{\frac{1}{2}\left[\langle Q_Tx, x\rangle_H +\left\langle \overline Q_T\left(x-S_T\int_H\xi m(T,d\xi)\right), \ x-S_T\int_H\xi m(T,d\xi)\right\rangle_H\right],}
\end{cases}
\end{small}
\end{equation}
and
\begin{equation}\label{FP:bis}
\begin{cases}
\displaystyle{m_{t}(t,dx)- \frac{1}{2}\mbox{Tr} [\sigma \sigma^{*}D^2m(t,dx)]+\mbox{div}\left(m(t,dx)\left(Ax-BR^{-1}B^*Dv\right)\right)=0,}\\\\
m(0,dx)=m_0(dx),
\end{cases}
\end{equation}
where $v:[0,T]\times H\to\R$ and $m:[0,T]\to \mathcal{P}(H)$ \textcolor{blu}{and the superscript $^*$ denotes the adjoint of a given operator}.
We now show, in an informal way, how we can deduce
the system \eqref{Riccati:Pintro}-\eqref{eq:rintro}-\eqref{eq:zintro}
from the HJB-FPK system \eqref{HJB:bis}-\eqref{FP:bis},

First, we guess solutions $v(t,x)$ to HJB \eqref{HJB:bis} in quadratic form:
\begin{equation}\label{v:quadratic}
v(t,x)=\frac{1}{2}\langle P(t)x,x\rangle_{H}+\langle r(t), x\rangle_{H}+s(t),
\end{equation}
where
$$
P:[0,T]\to \Sigma^{+}(H), \ \ \ \ r:[0,T]\to H, \ \ \ s:[0,T]\to\R.$$
Since
\begin{equation}\label{gradientsv}
Dv(t,x)=P(t)x+r(t), \quad D^2v(t,x)=P(t),
\end{equation}
the FP equation \eqref{FP:bis} becomes
\begin{equation}\label{FP}
\begin{small}
m_{t}(t,dx) =\frac{1}{2}\mbox{Tr}\Big[\sigma\sigma^{*}D^2m(t,dx)\Big]-\mbox{\large{div}}\Big(m(t,dx)\left((A-BR^{-1}B^*P(t))x-BR^{-1}B^*r(t)\right)\Big).
\end{small}
\end{equation}
Consider the following formal integration by parts formulas for a scalar function $\varphi:H\to\R$ and a $H$-valued function $w:H\to H$:
$$
\int_{H} \mbox{div}\bigg(w(x) m(t,dx)\bigg)\varphi (x)= -\int_{H} \langle w(x),D\varphi(x)\rangle_{H} \,\,m(t,dx),
$$
and
\begin{align*}
&\int_{H} \mbox{Tr}\big[\sigma\sigma^*D^{2}m(t,dx)\big]\varphi (x)= \int_{H} \mbox{div}\big[\sigma\sigma^*D m(t,dx)\big]\varphi (x)\\
&= \int_{H} \langle \sigma^*D m(t,dx), \sigma^{*}D\varphi (x)\rangle_{H} =   \int_{H}  m(t,dx)\,\, \mbox{div}(\sigma \sigma^{*}D\varphi (x)) =
 \int_{H} \mbox{Tr} \big[\sigma\sigma^{*}D^{2}\varphi(x)\big]\,\, m(t,dx).
\end{align*}
Let $\{{e}_{k}\}_{k\in\mathbb{N}}$ be an orthonormal basis of  $H$ and set the functions $$\pi_k\, : \, H\to \R, \ \ \ \ z\, : \, [0,T]\to H, \ \ \ \ \ z_k\, : \, [0,T]\to H,$$ as
$$\pi_{k}(x):=\langle x,{e}_{k}\rangle_{H},
\ \ \ \ z(t):=\int_H x\,m(t,\textcolor{blu}{dx)},
\ \ \ \
z_k(t):= \langle z(t), {e}_{k}\rangle_{H} =  \int_{H} \pi_{k}(x)m(t,dx).
$$
Using the above formulas and the fact
that $$D(\pi_{k})(x)\equiv\textcolor{blu}{e}_{k}, \ \ \ D^{2}(\pi_{k})(x)\equiv \mathbf{0}_{\mathcal{L}(H)},$$
we get, from \eqref{FP},
\begin{align*}
z_{k}'(t)& =\frac{d}{dt}  \int_{H} \pi_{k}(x)m(t,dx)=\int_{H} \pi_{k}(x)\partial_{t} m(t,dx) \\
&=\int_{H}  \pi_{k}(x) \left(\frac{1}{2}\mbox{Tr}\big[\sigma\sigma^{*}D^2m(t,dx)\big] -\mbox{div}\left(m(t,dx)\left((A-BR^{-1}B^*P(t))(x)-BR^{-1}B^*r(t)\right)\right)\right) \\
&=\int_{H} \langle (A-BR^{-1}B^*P(t))x-BR^{-1}B^*r(t),\, \mathbf{e}_{k}\rangle_{H\,\,} m(t,dx).
\end{align*}
Summing up over $k$, we  get
\begin{align*}
z'(t)&=\sum_{k=1}^{\infty} z'_k(t)e_{k}=\int_{H}  \left((A-BR^{-1}B^*P(t))x-BR^{-1}B^*r(t)\right)  m(t,dx)\\& = \left(A-BR^{-1}B^*P(t)\right) \int_{H}x m(t,dx)- BR^{-1}B^*r(t) \int_{H}m(t,dx)\\
&= (A-BR^{-1}B^*P(t))z(t)-BR^{-1}B^*r(t).
\end{align*}
Hence,
the FP equation reduces to    the following  $H$-valued abstract ODE for $z$:
\begin{equation}\label{eq:z}
\begin{cases}
\displaystyle{z'(t)=(A-BR^{-1}B^*P(t))z(t)-BR^{-1}B^*r(t),}\\\\
z(0)=z_0 :=\displaystyle{\int_{H} x\,\,m_0(dx)} \in H.
\end{cases}
\end{equation}
Moreover, plugging the structure \eqref{v:quadratic} into HJB  \eqref{HJB:bis} and  considering \eqref{gradientsv}, we get for $P(t)$ the backward Riccati equation
\begin{equation}\label{Riccati:P}
\begin{cases}
\displaystyle{P'(t)+P(t)A+A^*P(t)-P(t)BR^{-1}B^*P(t)+Q+\overline Q=0,}\\\\
P(T)=Q_T+\overline Q_T,
\end{cases}
\end{equation}
for $r$ the backward equation
\begin{equation}\label{eq:r}
\begin{cases}
\displaystyle{r'(t)+ (A^*-P(t)BR^{-1}B^*)r(t)-\overline Q Sz(t)=0},\\\\
r(T)=-\overline Q_TS_Tz(T),
\end{cases}
\end{equation}
and the following explicit expression for $s$ in terms of $P,z,r$:
\begin{align}\label{eq:s}
s(t)&=\frac{1}{2}\left\langle\overline Q_TS_Tz(T), S_Tz(T)\right\rangle_H\\
&\ \ +\int_t^T\left[\frac{1}{2}\mbox{Tr}[\sigma\sigma^*P(s)]-\frac{1}{2}\langle BR^{-1}B^*r(s), r(s)\rangle_H+\frac{1}{2}\langle \overline Q Sz(s), Sz(s)\rangle_H\right]ds.\nonumber
\end{align}
With regard to the last four equation written, we notice that the only coupled are the ones for $r$ and $z$, that is \eqref{eq:r} and \eqref{eq:z}. The system formed by these two equations is forward-backward and can be considered as the core reduction of the original HJB--FP system.
\medskip

\blu{We now introduce the concept of a solution that will be applied to the various equations, specifically the so-called \textit{mild solutions}. Since the operator $A$ may be unbounded, the notion of classical solutions (i.e., \( C^1 \) solutions) is not suitable in this context. To overcome this limitation, the infinite-dimensional literature (see \cite{DPB}) employs weaker notions of solutions for ODEs. The concept of mild solutions presented here is based on a generalization of the finite-dimensional \textit{variation of constants formula}.
}
\begin{definition}
\label{df:solsyst}  Denote by  $C_s([0,T];\Sigma^{+}(H))$  the space of strongly continuous operator-valued functions $f:[0,T]\to \Sigma^{+}(H)$, i.e., such that $t\mapsto f(t)x$ is continuous for each $x\in H$.

\begin{enumerate}[(i)]
\item
We say that $P\in  C_s([0,T];\Sigma^{+}(H))$ solves the Riccati equation \eqref{Riccati:P} in mild sense if, for all $x\in H$, $t\in[0,T]$,
\begin{eqnarray}\label{eq:mildP}
P(t)x&=& e^{(T-t)A^{*}}(Q_T+\overline Q_T)e^{(T-t)A}x+\int_{t}^{T} e^{(s-t)A^{*}}(Q+\overline Q)e^{(s-t)A}x\,ds\nonumber\\
&&-\int_{t}^{T}e^{(s-t)A^{*}}(P(s)B R^{-1}B^*P(s))e^{(s-t)A}x\,ds.
\end{eqnarray}
\item Given $P\in  C_s([0,T];\Sigma^{+}(H))$, we say that $(z,r)\in C([0,T];H^{2})$ solves  the forward-backward system \eqref{eq:z}-\eqref{eq:r} in mild sense if, for all $t\in[0,T]$,
\begin{equation}\label{eq:mildz}
z(t)=e^{t A}z_{0}-\int_{0}^{t} e^{(t-s)A} BR^{-1}B^*P(s)z(s) ds- \int_0^{t}e^{(t-s)A}BR^{-1}B^{*}r(s)ds,
\end{equation}
and
\begin{equation}\label{eq:mildr}
r(t)= e^{(T-t)A^{*}}(-\overline Q_{T}S_{T}z(T))-\int_t^{T} e^{(s-t)A^{*}}P(s)BR^{-1}B^{*}r(s)ds
-\int_t^{T}e^{(s-t)A^{*}}\overline Q S z(s)ds.
\end{equation}
\end{enumerate}
\end{definition}
Given the above definitions, we can now provide the following.
\begin{definition}[LQM mild solution to MFG]
\label{df:solMFG}
We say that a $4$-uple $$(P,r,z,s)\in C_s([0,T];\Sigma^{+}(H))\times C([0,T];H^{2})\times C([0,T];\R)$$ is a Linear-Quadratic-Mean (LQM) mild solution to the MFG system \eqref{HJB:bis}-\eqref{FP:bis}   if
\begin{enumerate}[(i)]
\item $P$ solves the Riccati equation \eqref{Riccati:P} in mild sense;
\item The couple $(r,z)$ solves the forward-backward system \eqref{eq:z}-\eqref{eq:r} in mild sense, with $P$ as in item (i);
\item $s$ is given by the expression \eqref{eq:s}, with $P,r,z$  as in items (i)-(ii).
\end{enumerate}
\end{definition}
The following remark will be used in the proof of existence.
\begin{remark}\label{rem:weak-mild}
Other concept of solutions to the above equations may be considered. Indeed, given $P \in C_s([0,T];\Sigma^{+}(H))$,  mild solutions to \eqref{eq:z}-\eqref{eq:r} defined as in \eqref{eq:mildz}-\eqref{eq:mildr} are equivalent to \emph{weak solutions} to the same equations (see \cite[Part II, Ch.\,1, Lemma\,3.2 and Prop.\,3.4]{DPB}); that is,  for all  $\phi \in D(A^*)$, $\psi\in D(A)$, and $t\in[0,T]$,

\begin{align}\label{eq:weak}
\langle \phi, z(t)\rangle_H\nonumber
=&\ \langle \phi, z_0\rangle_H +\int_0 ^{t} \langle A^*\phi, z(s)\rangle_H\,ds-\int_{0}^{t}\langle P(s)BR^{-1}B^*\phi, z(s)\rangle_H\,ds\\& -\int_{0}^{t}\langle \phi, BR^{-1}B^*r(s)\rangle_H\,ds,
\end{align}
and
\begin{align}\label{eq:weak}
\langle \psi, r(t)\rangle_H =\nonumber&\ \langle \psi, -\overline Q_TS_Tz(T)\rangle_H +\int_t ^{T} \langle A\psi, r(s)\rangle_H\,ds\\& -\int_{t}^{T}\langle \psi, P(s)BR^{-1}B^*r(s)\rangle_H\,ds-\int_{t}^{T}\langle \psi, \overline Q S z(s),\phi\rangle_H\,ds.
\end{align}

\end{remark}

\section{Existence of solutions to the MFG system}

Let $M\geq 1$ and $\omega\in \R$ be such that (see \cite[Part II, Ch.\,1, Cor.\,2.1]{DPB})
\begin{equation}\label{est:semi}
\|e^{tA}\|_{\mathcal{L}(H)},  \  \|e^{tA^*}\|_{\mathcal{L}(H)}\leq Me^{\omega t}.
\end{equation}
The Riccati equation \eqref{Riccati:P} is uncoupled and may be studied autonomously.
\begin{proposition}\label{lem:boundriccati}
The  Riccati equation
\eqref{Riccati:P}
admits a unique mild solution $P \in C_s([0,T]; \Sigma^+(H))$. Moreover,
\begin{equation}
\label{supP}
\sup_{t \in [0,T]}\|P(t)\|_{\mathcal{L}(H)}\leq M^2e^{2\omega^{+} T}(\|Q_T+\overline Q_T\|_{\mathcal{L}(H)}+T\|Q+\overline Q\|_{\mathcal{L}(H)}).
\end{equation}
\end{proposition}
\begin{proof}
For the existence and uniqueness of a mild solution in $C_s([0,T]; \Sigma^+(H))$ see \cite{DPB}, Part IV, Chapter 2, Theorem $2.1$ for the case $R=I$. The proof can be modified to cover also our case.\footnote{See also, for the general case,  the paper \cite{CP74} Theorem $4.1$ and Theorem $4.2$  (with $\mathcal{W}=Q+\overline Q$ and $\mathcal{G}=Q_T+\overline Q_T$) or the book \cite{CP78}, Chapter $4$, Lemma $4.3$ and Lemma $4.6$ (with $M=Q+\overline Q$ and $G=Q_T+\overline Q_T$), dealing however with other concepts of solutions.}
Let us  prove \eqref{supP}. Let $x\in H$. By definition of mild solution,    we have
\begin{align*}
\langle P(t)x,x\rangle_H=& \ \langle e^{(T-t)A^*}(Q_T+\overline Q_T)e^{(T-t)A}x, x\rangle_H+\int_t^T\langle e^{(s-t)A^*}(Q+\overline Q)e^{(s-t)A}x,x\rangle_H\,ds\\&-\int_t^T\langle R^{-1}B^*P(s)e^{(s-t)A}x,B^*P(s)e^{(s-t)A}x\rangle_U\,ds.
\end{align*}
Since $R^{-1}$ is nonnegative,
 we get
\begin{align*}
\langle P(t)x,x\rangle_H& \leq\langle e^{(T-t)A^*}(Q_T+\overline Q_T)e^{(T-t)A}x, x\rangle_H+\int_t^T\langle e^{(s-t)A^*}(Q+\overline Q)e^{(s-t)A}x,x\rangle_Hds\\
 &\leq \|e^{(T-t)A^*}(Q_T+\overline Q_T)e^{(T-t)A}\|_{\mathcal{L}(H)}\,|x|^{2}_H+\int_t^T\|e^{(s-t)A^*}(Q+\overline Q)e^{(s-t)A}\|_{\mathcal{L}(H)}\,|x|^{2}_Hds.\\
 &\leq M^2e^{2\omega^{+} T}(\|Q_T+\overline Q_T\|_{\mathcal{L}(H)}+T\|Q+\overline Q\|_{\mathcal{L}(H)})\,|x|^{2}_H.
 \end{align*}
Recalling that $P\in \Sigma^{+}(H)$,
we therefore conclude because of the equality
$$
\|\mathcal{Q}\|_{\mathcal{L}(H)}=\sup_{|x|_H=1}|\langle \mathcal{Q} x,x\rangle_{H}|, \ \ \ \ \forall \mathcal{Q}\in \Sigma(H).
$$
\end{proof}
Let $(A_n)_{n\in\mathbb{N}}\subset \mathcal{L}(H)$ be the Yosida approximations of the operator $A$ defined as
\begin{equation}\label{Yosida}
A_{n}=n^2R(n,A)-nI,
\end{equation}
where $R(n,A)$ is the resolvent operator of $A$.
 For future reference, we recall some properties concerning them (see \cite[p.\,102]{DPB}): we have
 \begin{equation}\label{sexiest}
M_T:=\sup_{t\in[0,T],\,n\in\mathbb{N}}\|e^{tA_n}\|_{\mathcal{L}(H)}= \sup_{t\in[0,T],\,n\in\mathbb{N}}\|e^{tA_n^{*}}\|_{\mathcal{L}(H)}<\infty,
\end{equation}
\begin{equation}\label{semiconv}
e^{tA_{n}}\to e^{tA} x,  \ \ \ e^{tA^{*}_{n}}\to e^{tA^{*}} x, \ \ \ \ \forall t\in[0,T], \ \forall x\in H,
\end{equation}
and
\begin{equation}\label{A*conv}
 A_nx \to Ax  \ \ \ \forall x\in D(A), \ \ \ \ \ \ \ \ \  A^*_nx \to A^*x \ \ \ \ \forall x\in D(A^*).
\end{equation}
\medskip

In order to prove existence for the system \eqref{eq:z}-\eqref{eq:r}, we need the following.

\begin{assumption}\label{ass:etan}

Let $P \in C_s([0,T]; \Sigma^+(H))$ be the mild solution to the Riccati equation \eqref{Riccati:P} provided by Proposition \ref{lem:boundriccati} and let $(A_{n})$ be the sequence of operators defined by \eqref{Yosida}.
\smallskip
 \begin{itemize}
\item[(H1)] For $n\in\mathbb{N}$, the Riccati equations
\begin{equation}\label{eqn:riccatieta}
\begin{cases}
 \eta_n'(t)=&(P(t)BR^{-1}B^*-A_n^*)\eta_n(t)-\eta_n(t)(A_n-BR^{-1}B^*P(t)) \\&+\overline Q S+\eta_n(t)BR^{-1}B^*\eta_n(t),\\\\
 \eta_n(T)=&-\overline Q_T S_T,
 \end{cases}
\end{equation}
admit strict solutions in the space $C^{1}_{s}([0,T];\Sigma(H))$; that is, there exist $\eta_{n}:[0,T]\to \Sigma(H)$ such that, for all $n\in\mathbb{N}$,
\medskip

\begin{enumerate}[(i)]
\item $ \eta_n(T)=-\overline Q_T S_T$;

\item  the map $[0,T]\to H$,  $t\mapsto \eta_{n}(t)x$ is differentiable for each $t\in [0,T]$ and $x\in H$;

\item  for all $t\in [0,T]$, $x\in H$, it holds the equality
\begin{align*}
\frac{d}{dt}\eta_n(t)x
=&\ (P(t)BR^{-1}B^*-A_n^*)\eta_n(t)x-\eta_n(t)(A_n-BR^{-1}B^*P(t))x \\&+\overline Q Sx+\eta_n(t)BR^{-1}B^*\eta_n(t)x.
\end{align*}
\end{enumerate}
\item[(H2)] Under (H1), we have
\begin{equation}\label{estunifeta}
\sup_{t\in [0,T], n\in\mathbb{N}}\|\eta_{n}(t)\|_{\mathcal{L}(H)} <\infty.
\end{equation}
\end{itemize}
\end{assumption}
\begin{remark} An inspection of the use of Assumption \ref{ass:etan} in the proof of Theorem \ref{Th:existence} shows that the former  may be relaxed by requiring that (H1) holds just definitively in $n$ and by replacing (H2) with
$$
\liminf_{n\to \infty}\sup_{t\in [0,T]}\|\eta_{n}(t)\|_{\mathcal{L}(H)} <\infty.
$$
\end{remark}
We will discuss the validity of Assumption \ref{ass:etan} in the Appendix (see Proposition \ref{prop:positive}). Here, we only notice that it is satisfied, in particular, without further assumptions if $T$ is small enough (cf. Remark \ref{rem:tsmall}).
We turn now to our existence result.

\begin{theorem}[Existence]\label{Th:existence}
Let Assumption \ref{ass:etan} hold. Then,
there exists a LQM mild solution to MFG.
\end{theorem}
\begin{proof}
The existence (and uniqueness) of mild solutions to \eqref{Riccati:P} in $C_s([0,T], \Sigma(H)^+)$ is provided by Proposition \ref{lem:boundriccati}.

In order to show the existence of solutions to
\eqref{eq:z}-\eqref{eq:r},  we proceed in several steps. First,  we consider approximating versions of \eqref{eq:z}-\eqref{eq:r} with the Yosida approximations $A_n$ in place of $A$ and find a solution $(z_n,r_n)$ by decoupling the system inspired by the finite dimensional case (see \cite{BeFrYa}, Chapter 6). Second, we prove estimates, uniform in $n$, on the constructed solution $(z_n,r_n)$ of the approximating system. Third, we pass to the limit as $n\to \infty$ by applying  Ascoli-Arzel\'a's Thorem  in metric spaces to get  the (weak) convergence of  $(z_n, r_n)$ to a couple $(z,r)$. Fourth, we show that the limit  $(z,r)$  solves \eqref{eq:z}-\eqref{eq:r}. \\

\begin{enumerate}[\emph{Step 1.}]
\item
 We consider the approximating systems
\begin{equation}\label{eq:zetayosida}
\begin{cases}
\displaystyle{z_n'(t)=(A_n-BR^{-1}B^*P(t))z_n(t)-BR^{-1}B^*r_n(t),}\\\\
z_n(0)=z_0,
\end{cases}
\end{equation}
and
\begin{equation}\label{eq:ryosida}
\begin{cases}
\displaystyle{r_n'(t)+ (A_n^*-P(t)BR^{-1}B^*)r_n(t)-\overline Q Sz_n(t)=0},\\\\
r_n(T)=-\overline Q_TS_Tz_n(T).
\end{cases}
\end{equation}
For the ODEs considered in this step, we use the notion of so called \emph{strict solutions}; that is, functions in the space $C^{1}([0,T];H)$ which satisfy the ODEs in the classical sense for all $t\in[0,T]$(\footnote{This corresponds to the terminology of \cite{DPB} (Definition 3.1(i), p.\,129), combined with Proposition 3.3(ii), p.\,133. Cf. also the definition given in Assumption \ref{ass:etan} for the operator-valued Riccati equation.}).
To decouple the system, let us assume that a  solution in this sense to  \eqref{eq:zetayosida}-\eqref{eq:ryosida}
exists in the form
 \begin{equation}\label{eq:rprod}
 r_n(t)=\eta_n(t)z_n(t),
 \end{equation} with $\eta_n\in C_{s}^{1}([0,T];{\Sigma}(H))$, $\eta_{n}(T)=-\overline Q_{T}S_{T}$,
and $z_n \in C^1([0,T],H)$. \textcolor{blu}{We can decouple the system in this manner because the solutions of the approximating system \eqref{eq:zetayosida}-\eqref{eq:ryosida} are strict. This is due to the fact that the operators $A_n$ are bounded, allowing us to apply the classical product rule to differentiate \eqref{eq:rprod}. This is not feasible in the original Riccati system, as in that case, the solutions cannot generally be assumed to be strict.}

Imposing this structure, one  formally  gets
\begin{equation}\label{eq:r1}
r'_n(t)=\eta_n'(t)z_n(t)+\eta_n(t) z'_n(t),
\end{equation}
and,
plugging into \eqref{eq:ryosida}, one gets
\begin{equation}\label{eq:r2}
\eta_n(t)z_n'(t)=(P(t)BR^{-1}B^*-A_n^*)\eta_n(t)z_n(t)+\overline QS z_n(t)- \eta'_n(t)z_n(t).
\end{equation}
On the other hand, plugging into \eqref{eq:zetayosida}, we get
\begin{equation}\label{eq:r3}
\eta_n(t)z_n'(t)=\eta_n(t)(A_n-BR^{-1}B^*P(t))z_n(t)-\eta_n(t)BR^{-1}B^*\eta_n(t)z_n(t).
\end{equation}
Equating  the two expressions above, we get the following equation for $\eta_n$:
\begin{equation*}\label{eq:eta1}
 \eta_n'(t)=(P(t)BR^{-1}B^*-A_n^*)\eta_n(t)+\overline Q S-\eta_n(t)(A_n-BR^{-1}B^*P(t))+\eta_n(t)BR^{-1}B^*\eta_n(t).
\end{equation*}
 In this way, we
 have disentagled the system  into
 \begin{equation}\label{eq:eta2}
\begin{cases}
 \eta_n'(t)=(P(t)BR^{-1}B^*-A_n^*)\eta_n(t)-\eta_n(t)(A_n-BR^{-1}B^*P(t))\\
 \ \ \ \ \ \ \ \  \ \ \ \ +\overline Q S+\eta_n(t)BR^{-1}B^*\eta_n(t),\\\\
 \eta_n(T)=-\overline Q_T S_T,
 \end{cases}
\end{equation}
and
\begin{equation}\label{eq:zetayosidabis}
\begin{cases}
\displaystyle{z_n'(t)=(A_n-BR^{-1}B^*P(t))z_n(t)-BR^{-1}B^*\eta_n(t)z_n(t),}\\\\
z_n(0)=z_0.
\end{cases}
\end{equation}
By Assumption  \ref{ass:etan},
 \eqref{eq:eta2}
is a Riccati equation admitting a strict solution $\eta_{n}$.
Plugging its expression into \eqref{eq:zetayosidabis}, one gets a corresponding unique strict solution  $z_{n}$ to the latter. Then,  defining $r_{n}$ as in  \eqref{eq:rprod}, one may use the formal computations \eqref{eq:r1}--- \eqref{eq:eta2} and conclude that  $r_{n}$ so defined  is actually a strict solution to \eqref{eq:ryosida}. Hence, the couple $(z_n,r_n)$ so constructed is a strict solution to the coupled system \eqref{eq:zetayosida}-\eqref{eq:ryosida}.\\




\item[\emph{Step 2.}]
Let $z_{n},r_{n},\eta_{n}$ be the functions defined in Step 1. We are going to give estimates uniform in $n$ for $z_{n},r_{n}$.  In the following, $C$ will be a positive constant,  depending on the data of the problem but independent of $n\in \mathbb{N}$, which may change from line to line.\\

%
%

i) \emph{(Estimates on $z_n$)}
Clearly, being $z_n$ a strict solution to \eqref{eq:zetayosidabis}, it is also a mild solution to the same equation; that is, for all $t\in[0,T]$,
\begin{equation}\label{eq:mildzn}
z_{n}(t)=e^{t A_{n}}z_{0}-\int_{0}^{t} e^{(t-s)A_{n}} BR^{-1}B^*P(s)z_{n}(s) ds- \int_0^{t}e^{(t-s)A_{n}}BR^{-1}B^{*}\eta_{n}(s)z_{n}(s)ds.
\end{equation}
 Then,
using
\eqref{supP}, \eqref{estunifeta} and
\eqref{sexiest},
 we get
$$
|z_n(t)|_H\leq C\left(|z_0|_H+\int_0^t
|z_n(s)|_Hds\right),  \ \ \ \forall  t\in[0,T].
$$
By the
 Gronwall's Lemma, we then get
\begin{equation}\label{eqn:boundznbis}
|z_n(t)|_H\leq C \ \ \ \forall t\in[0,T].
\end{equation}
Set now
\begin{equation}\label{deftildez}
\tilde{z}_{n}(t):=z_{n}(t)-e^{tA_n}z_{0}, \ \ \ \forall  t\in[0,T].
\end{equation}
By \eqref{eqn:boundznbis} and \eqref{sexiest}, we get
\begin{equation}\label{eqn:boundzn}
|\tilde{z}_n(t)|_H\leq C,  \ \ \ \forall  t\in[0,T].
\end{equation}
Then, still employing \eqref{eq:mildzn} and using \eqref{eqn:boundznbis}, we also get
the estimate
\begin{equation}\label{est:z}
|\tilde{z}_n(t)-\tilde{z}_n(s)|_H\leq C|t-s|,\ \ \ \ t,s\in [0,T].
\end{equation}

%
ii) \emph{(Estimates on $r_n$)}
For the sequence $\{r_{n}:[0,T]\to H\}_{n\in \mathbb{N}}$, we proceed similarly. We have, by the mild formulation
\begin{align}\label{eqn:mildrn}
r_n(t)=&\ e^{(T-t)A_n^*}(-\overline Q_T S_T z_n(T))-\int_t^Te^{(s-t)A_n^*}P(s)BR^{-1}B^*r_n(s)ds\nonumber\\
&-\int_t^Te^{(s-t)A_n^*}\overline Q S z_n(s)ds,   \ \ \ \forall  t\in[0,T].
\end{align}
 Then,
using
\eqref{supP}, \eqref{sexiest} and \eqref{eqn:boundznbis},
 we get
$$
|r_n(t)|_H\leq C\left(|z_{n}(T)|_H+\int_t^T
|r_n(s)|_Hds\right),   \ \ \ \forall  t\in[0,T].
$$
By  \eqref{eqn:boundznbis} and the
 Gronwall's Lemma, we then get
\begin{equation}\label{eqn:boundznbisr}
|r_n(t)|_H\leq C, \ \ \ \forall t\in[0,T].
\end{equation}
Set now
\begin{equation}\label{deftilder}
\tilde{r}_{n}(t):=r_{n}(t)-e^{(T-t)A_n^{*}}(-\overline Q_T S_T z_n(T)), \ \ \ t\in[0,T].
\end{equation}
By \eqref{eqn:boundznbisr}, \eqref{eqn:boundznbis}, and \eqref{sexiest}, we get
\begin{equation}\label{eqn:boundrn}
|\tilde{r}_n(t)|_H\leq C.
\end{equation}
Then, still employing \eqref{eqn:mildrn} and using \eqref{supP}, \eqref{sexiest}, \eqref{eqn:boundzn}, and \eqref{eqn:boundznbisr}, we  get
the estimate
\begin{equation}\label{est:z}
|\tilde{r}_n(t)-\tilde{r}_n(s)|_H\leq C|t-s|,\ \ \ \ t,s\in [0,T].
\end{equation}\medskip
%

\smallskip

\item[\emph{Step 3.}] We are going to prove that the sequence $\{(z_n,r_{n})\}_{n\in\mathbb{N}}$ admits a subsequence converging, in a suitable sense, to a limit $(z,r)$. Let $\rho>0$ be such that $\tilde{z}_{n}$ and $\tilde{r}_n$ defined in the previous step take value in $$\mathcal{B}_{\rho}:=\{x\in H: \ |x|_H\leq \rho\}$$ for all $n\in\mathbb{N}$ (see \eqref{eqn:boundzn}-\eqref{eqn:boundrn}).
The weak topology of the separable Hilbert space $H$ is metrizable on the ball
$\mathcal{B}_{\rho}$ (see, e.g.,  \cite[Th.\,3.29]{BR}). Precisely,  letting $\{a_n\}_{n\in\mathbb{N}}$ be a dense subset of $\mathcal{B}_{\rho}$, the distance
$$
d(x,y)=\sum_{n\in\mathbb{N}} 2^{-(n+1)} |\langle x-y,a_{n}\rangle_H|, \ \ \ x,y\in \mathcal{B}_{\rho},
$$
induces the weak topology on $\mathcal{B}_{\rho}$.
Notice that
\begin{equation}\label{distnorm}
d(x,y)\leq \rho |x-y|_H, \ \ \ x,y\in \mathcal{B}_{\rho}.
\end{equation}
Then, given Step 2 and \eqref{distnorm}, we may  apply Ascoli-Arzel\`a's Theorem in the space $C([0,T]; (\mathcal{B}_{\rho},d))$ and get the existence of a subsequence, that with abuse notation we still denote by $\{(\tilde{z}_{n},\tilde{r}_n)_{n}\}$, and of a couple $\tilde{z},\tilde{r}\in C([0,T];(\mathcal{B}_{\rho},d))$ such that

%
%
%

\begin{equation}\label{weaker}
\lim_{n\to\infty} \sup_{[0,T]} \big(d(\tilde{z}_n(t),\tilde{z}(t))+d(\tilde{r}_n(t),\tilde{r}(t))\big)=0,
\end{equation}
In particular, denoting by $\rightharpoonup$  the weak convergence in $H$,
\begin{equation}\label{weakconvbis}
\tilde{z}_{n}(t) \rightharpoonup \tilde{z}(t), \ \ \ \tilde{r}_n(t) \rightharpoonup \tilde{r}(t), \ \ \ \ \forall t\in[0,T].
\end{equation}
On the other hand, by \cite{DPB}, Part II, Chapter 1, Theorem 2.5 we have
\begin{equation}\label{est:yosida}
\lim_{n\to\infty}\sup_{t\in[0,T]} |(e^{tA}-e^{tA_n})x|_H=0, \ \ \ \forall x\in H.
\end{equation}
Therefore, from \eqref{weakconvbis} and \eqref{est:yosida}, it follows that
\begin{equation}\label{weakconv}
z_n(t)=\tilde{z}_{n}(t) + e^{tA_n}z_0 \ \ {\rightharpoonup} \ \ \tilde{z}(t)+ e^{tA}z_0=:z(t)\ \ \ \ \forall t\in[0,T].
\end{equation}

We may argue similarly for $r_{n}$ as follows.
First of all, we  note that by Lemma \ref{lem:convT} with
$$F_n=e^{(T-t)A_n^*}, \ \ \ x_n=-\overline Q_TS_Tz_n(T),$$ we obtain
\begin{equation}
\label{eqzz}
e^{(T-t)A^{*}_n}(-\overline Q_T S_T z_n(T)) \ \ \rightharpoonup  \ \
e^{(T-t)A^{*}} (-\overline Q_T S_T z(T)), \ \ \ \ \forall t\in[0,T].
\end{equation}
Accounting for \eqref{weakconvbis}, it follows
\begin{align}\label{weakconvr}
r_n(t)&=
\tilde{r}_{n}(t) + e^{(T-t)A^{*}_n}(-\overline Q_T S_T z_n(T)) \nonumber \\
&\ \ \ \ \ \ \rightharpoonup \ \
\tilde{r}(t)+ e^{(T-t)A^{*}} (-\overline Q_T S_T z(T))=:r(t)\ \ \ \forall t\in[0,T].
\end{align}
Notice  that $z,r\in C([0,T];H_{w})$,  where $H_{w}$ denotes the space $H$ endowed with the weak topology.
Hence, we have proved that it exists a subsequence of $\{(z_{n},r_{n})_{n}\}$, still labeled in the same way, and $(z,r)\in C([0,T];H_{w})$ such that
\begin{equation}\label{weakconvbis2}
{z}_{n}(t) \rightharpoonup {z}(t), \ \ \ {r}_n(t) \rightharpoonup {r}(t), \ \ \ \ \forall t\in[0,T].
\end{equation}
\medskip
\item[\emph{Step 4.}]
Let us show that $z$ defined by Step 3 solves \eqref{eq:mildz}.
Due to Remark \ref{rem:weak-mild}, $z_n$ is also a weak solution to \eqref{eq:zetayosida}, that is
for all $\phi \in D(A^*)$, $t\in[0,T]$,
 \begin{align}\label{weaksol}
\langle \phi, z_n(t)\rangle_H=&\ \langle \phi, z_0\rangle_H+\int_0^t\langle A_n^*\phi, z_n(s)\rangle_H\,ds -\int_0^t\langle P^*BR^{-1}B^*\phi, z_n(s)\rangle_H\,ds \\
&-\int_0^t\langle \phi, BR^{-1}B^*r_n(s)\rangle_H\,ds\nonumber.
\end{align}
We want now to take the limit as $n\to\infty$.
By \eqref{weakconvbis2}, we have for each $t,s\in[0,T]$,
\begin{equation}\label{convex}
\begin{cases}\langle \phi, z_n(t)\rangle_H \rightarrow \langle \phi, z(t)\rangle_H,\medskip\\
 \langle P^*BR^{-1}B^*\phi, z_n(s)\rangle_H\rightarrow \langle P^*BR^{-1}B^*\phi, z(s)\rangle_H,\medskip\\
\langle \phi, BR^{-1}B^*r_n(s)\rangle_H \rightarrow \langle \phi, BR^{-1}B^*r(s)\rangle_H.
\medskip
\end{cases}
\end{equation}
Moreover, for each $s\in [0,T]$,
\begin{equation}\label{AAn}
|\langle A_n^*\phi, z_n(s)\rangle_H -\langle A^*\phi, z(s)\rangle_H|\leq |\langle (A_n^*-A^*)\phi, z_n(s)\rangle_H|+|\langle A^*\phi, z_n(s)-z(s)\rangle_H|.
\end{equation}
Now, on the one hand, by \eqref{weakconvbis2}
\begin{equation}\label{AAn2}
|\langle A^*\phi, z_n(s)-z(s)\rangle_H |\rightarrow 0;
\end{equation}
on the other hand, by \eqref{eqn:boundzn}, we have
\begin{equation}\label{AAn3}
|\langle (A_n^*-A^*)\phi, z_n(s)\rangle_H|\leq |(A_n^*-A^*)\phi|_H|z_n(s)|_H\leq C|(A^*_n-A^*)\phi|_H\rightarrow 0,
\end{equation}
where the latter convergence follows from \eqref{A*conv}. Therefore, combining \eqref{AAn}, \eqref{AAn2}, and \eqref{AAn3},  we get, for each $t\in[0,T]$,
\begin{equation}\label{convA*}
\langle A_n^*\phi, z_n(t)\rangle_H \rightarrow \langle A^*\phi, z(t)\rangle_H.
\end{equation}
Noting that, definitively in $n$,
$$
\sup_{t\in[0,T]}|\langle A^*_n \phi, z_n(t)\rangle_H|\leq \sup_{[0,T]}|z_n(s)|_H|A_n^*\phi|_H\leq |A^*\phi|_H+1,
$$
we may use dominated convergence to pass to the limit  in \eqref{weaksol} and use \eqref{convex} and \eqref{convA*} to conclude that
\begin{align*}
\langle \phi, z(t)\rangle_H=& \ \langle \phi, z_0\rangle_H +\int_0^t\langle A_n^*\phi, z(s)\rangle_H\,ds \\&+\int_0^t\langle P^*BR^{-1}B^*\phi, z(s)\rangle_H\,ds -\int_0^t\langle \phi, BR^{-1}B^*r(s)\rangle_H\,ds.
\end{align*}
This says that $z$ is a weak solution to \eqref{eq:z}. By Remark \ref{rem:weak-mild}, it is also a mild solution to the same equation, i.e. solves \eqref{eq:mildz}.

In a similar way, one may prove that that $r$ defined by Step 3 solves \eqref{eq:mildr}, concluding the proof.
\end{enumerate}
\end{proof}
\section{Uniqueness of solutions}
In this section, we prove  two uniqueness results. We start with a result of this kind under the assumption that $T$ is small enough\footnote{Notice that the argument is based on the standard contraction fixed point, hence it would provide also existence.}.
Set
$$
C_{BR}:=\|BR^{-1}B^*\|_{\mathcal{L}(H)}, \quad C_{\overline{Q}_TS_T}:=\|\overline Q_TS_T\|_{\mathcal{L}(H)}, \quad C_{QS}:=\|QS\|_{\mathcal{L}(H)}.
$$
\blu{Notice that the proof of the following result also provide, at once, the existence for small time horizon, as it is based on the classic contraction argument (cf. also \cite{liu2024hilbert}).}
\begin{proposition}[Uniqueness for small time horizon]\label{prop:small}
Let $T>0$ be such that\footnote{Clearly, since
$$
\lim_{T \to 0} M^2(C_{\overline{Q}_TS_T}+C_{QS}T)TC_{BR}e^{2Me^{\omega T}C_{BR}\beta T+2\omega T}=0,
$$
there exists $T>0$ such that \eqref{Tsmall} holds.
}
\begin{equation}\label{Tsmall}
C_{T}:=M^2(C_{\overline{Q}_TS_T}+C_{QS}\,\,T)\,\,TC_{BR}e^{2Me^{\omega T}C_{BR}\beta T+2\omega T}\ <\ 1,
\end{equation}
with $M,\omega$ as in \eqref{est:semi}.
Then, the LQM mild solution to MFG is unique.
\end{proposition}
\begin{proof}
We consider the space $C([0,T];H)$ endowed with the usual sup-norm
$$|f|_{\infty}:=\sup_{[0,T]}|f(t)|_{H}.$$

 Consider the map
$$\Psi: C([0,T];H) \to C([0,T];H)$$
where  $\Psi(r)$ is the unique mild solution to \eqref{eq:z} --- that is, \eqref{eq:mildz} holds.
 Then,
 consider the map
 $$\Phi: C([0,T];H) \to C([0,T];H)$$
where $\Phi(z)$ is the unique mild solution to \eqref{eq:r} --- that is,  \eqref{eq:mildr} holds.

 By construction, $\hat \zeta \in C([0,T];H)$ is a fixed point of $\Phi\circ \Psi$ if and only if $(\hat \zeta, \Psi(\hat \zeta))$ is a solution to the coupled system \eqref{eq:mildz}-\eqref{eq:mildr}.
We are going to prove that
\begin{equation}\label{eqn:contraction}
|(\Phi\circ \Psi)(r_1)(\cdot)-(\Phi\circ \Psi)(r_2)(\cdot)|_\infty\leq C_{T} |r_1(\cdot)-r_2(\cdot)|_\infty,\quad \forall r_1, r_2 \in C([0,T]; H).
\end{equation}
Due  to \eqref{Tsmall}, this guarantees the (existence and) uniqueness of a fixed point by the Banach-Caccioppoli fixed point theorem.
We proceed with some estimates.
\begin{enumerate}
\item[\emph{Step 1.}] We prove a Lipschitz estimate for $\Psi$. We write
\begin{multline*}
|\Psi(r_1)(t)-\Psi(r_2)(t)|_H=|z_1(t)-z_2(t)|_H\\
\leq\left|\int_0^te^{(t-s)A}BR^{-1}B^*P(s)(z_1(s)-z_2(s))ds \right|_H+\left|\int_0^te^{(t-s)A}BR^{-1}B^*(r_1(s)-r_2(s))ds \right|_H.
    \end{multline*}
Now, recalling \eqref{supP} and setting
\begin{equation}\label{beta}
\beta:=M^2e^{2\omega T}(\|Q_T+\overline Q_T\|_{\mathcal{L}(H)}+T\|Q+\overline Q\|_{\mathcal{L}(H)}),
\end{equation}
we have
$$
\left|\int_0^te^{(t-s)A}BR^{-1}B^*P(s)(z_1(s)-z_2(s))ds\right|_H\leq Me^{\omega T}C_{BR}\beta\int_0^t|z_1(s)-z_2(s)|_Hds.
$$
Moreover,
$$
\left|\int_0^te^{(t-s)A}BR^{-1}B^*(r_1(s)-r_2(s)ds \right|_H\leq Me^{\omega T}TC_{BR}\,|r_1(\cdot)-r_2(\cdot)|_\infty.
$$
Then, combining the inequalities above and using  Gronwall's Lemma, we have
$$
|\Psi(r_1)(\cdot)-\Psi(r_2)(\cdot)|_\infty=|z_1(\cdot)-z_2(\cdot)|_\infty\leq MTC_{BR}\, e^{Me^{\omega T}C_{BR}\beta T+\omega T}\,\, |r_1(\cdot)-r_2(\cdot)|_\infty.
$$

\item[\emph{Step 2.}] We prove a Lipschitz estimate for $\Phi$. For $i=1,2$, set $y_i(t):=r_i(T-t)$. Then,
$$
\begin{cases}
    y_i'(t)=(A^*-P(T-t)BR^{-1}B^*)y_i(t)-QSz_i(T-t),\\
    y_i(0)=-Q_TS_Tz_i(T);
\end{cases}
$$
that is, by the mild formulation
\begin{align*}
y_i(t)=& \ -e^{tA^*}Q_TS_Tz_i(T)-\int_0^te^{(t-s)A^*}P(T-s)BR^{-1}B^*y_i(s)ds\\&-\int_0^te^{(t-s)A^*}QS z_i(T-s)ds.
\end{align*}
Then
\begin{align*}
&|y_1(t)-y_2(t)|_H\\&\leq \  Me^{\omega T}C_{\overline{Q}_TS_T} |z_1(\cdot)-z_2(\cdot)|_\infty+\left|\int_0^te^{(t-s)A^*}P(T-s)BR^{-1}B^*(y_1(s)-y_2(s))ds\right|_H\\&\ \ \ \ +\left|\int_0^te^{(t-s))A^*}QS(z_1(T-s)-z_2(T-s))ds\right|_H.
\end{align*}
Note that by \eqref{supP} and \eqref{beta} we have
$$
\left|\int_0^te^{(t-s)A^*}P(T-s)BR^{-1}B^*(y_1(s)-y_2(s))ds\right|_H\leq Me^{\omega T}\beta C_{BR}\int_0^t|y_1(s)-y_2(s)|_Hds
$$
and
$$
\left|\int_0^te^{(t-s))A^*}QS(z_1(T-s)-z_2(T-s))ds\right|_H\leq Me^{\omega T}C_{QS}T |z_1(\cdot)-z_2(\cdot)|_\infty.
$$
Therefore, by  Gronwall's Lemma,
\begin{align*}
|\Phi(z_1)(\cdot)-\Phi(z_2)(\cdot)|_\infty=&\ |r_1(\cdot)-r_2(\cdot)|_\infty\\ \leq &M(C_{\overline{Q}_TS_T}+C_{QS}T) \,\,e^{Me^{\omega T}\beta C_{BR}T+\omega T}\,\, |z_1(\cdot)-z_2(\cdot)|_\infty.
\end{align*}
\item[\emph{Step 3.}] We may combine the previous estimates and   get
\begin{align*}
&|(\Phi \circ \Psi)(r_1)(\cdot)-(\Phi \circ \Psi)(r_2)(\cdot)|_\infty\\&\leq M(C_{\overline{Q}_TS}+C_{QS}T)|\Psi(r_1)(\cdot)-\Psi(r_2)(\cdot)|_\infty \,\,e^{Me^{\omega T}\beta C_{BR}T+\omega T}\\ &\leq  M^2(C_{\overline{Q}_TS}+C_{QS}T)TC_{BR} \,e^{2Me^{\omega T}C_{BR}\beta T+2\omega T}\,\, |r_1(\cdot)-r_2(\cdot)|_\infty ,
\end{align*}
completing the proof.
\end{enumerate}
\end{proof}
Under some further assumptions, in particular  requiring the dissipativity of the operators $\overline Q S$ and  $\overline Q_T S_T$, we are able to prove that uniqueness holds true for large  time horizon. Notice that the assumption of dissipativity of the latter operators also ensures the validity of Assumption \ref{ass:etan} (see Proposition \ref{prop:positive}) and, in turn, the existence of a solution for large time horizon.
\begin{theorem}\label{Th:uniqueness}(Uniqueness for large time horizon)
Assume that:
\begin{enumerate}[(a)]
\item $  -\overline Q_T S_T, -\overline Q S \in \Sigma^{+}(H);$
\item The following implications hold true:
$$\langle \overline Q_TS_Tx,x\rangle=0  \ \Rightarrow \ \overline Q_TS_Tx=0, \ \ \ \ \  \langle \overline QSx,x\rangle=0  \ \Rightarrow \ \overline QSx=0.$$
\end{enumerate}
 Then, the LQM mild solution to MFG is unique.
\end{theorem}
\begin{remark}\label{rem:uniqueness}
%
Assumption (iii) \blu{on $-\overline{Q} S$ (resp., on $-\overline Q_T S_T$)} of the above Theorem \ref{Th:uniqueness} is satisfied if, for instance, one of the following conditions holds:
\begin{itemize}
    \item[(i)] \(-\overline{Q} S \in \Sigma^{++}(H)\) \blu{(resp., $-\overline Q_T S_T\in \Sigma^{++}(H)$)};
    \item[(ii)] \(-\overline{Q} S\) \blu{(resp., $-\overline Q_T S_T$)}  is a projection onto a closed subspace of \(H\).
\end{itemize}
\end{remark}
\begin{proof}  Clearly, it suffices to prove uniqueness of (mild) solutions to the coupled system \eqref{eq:z}-\eqref{eq:r}. For that,
let $(z_1,r_1)$ and $(z_2, r_2)$ be two solutions to the aforementioned system. Define
$$
\hat z:=z_1-z_2, \quad \hat r:=r_1-r_2.
$$
We are going to show that $(\hat z,\hat r)=(0, 0)$.
We split the proof in several steps. In the following,  $C$ will denote a positive constant, not depending on $n$, which may change from line to line.
\begin{itemize}
	\item[\emph{Step 1.}]
We notice that, given $\hat r$, we have that
 $\hat z$  is the unique mild solution to
\begin{equation}\label{eq:barz}
z'(t)=(A-BR^{-1}B^*P(t))z(t)-BR^{-1}B^*\hat r(t), \quad z(0)=0;
\end{equation}
similarly, given $\hat z$, we have that $\hat r$ is
 the unique  mild solution to
\begin{equation}\label{eqn:overliner}
r'(t)+(A^*-P(t)BR^{-1}B^*)r(t)-\overline Q S\hat z(t)=0, \quad r(T)=-\overline Q_TS_T\hat z(T).
\end{equation}
Clearly, by continuity of  $(z_i, r_{i})$, for $i=1,2$,
we have,
\begin{equation}\label{eqn:barrbounded}
|\hat r(t)|_H,  |\hat z(t)|_H \leq C \quad \forall t \in [0,T].
\end{equation}
The goal is now  to compute $\frac{d}{dt}\langle \hat z(t), \hat r(t)\rangle_{H}$. To achieve it, we pass through an approximation with the Yosida approximants $(A_{n})$ of the operator $A$ defined in \eqref{Yosida}.
\item[\emph{Step 2.}]

Let $\hat z_n$ be the unique strict solution to
\begin{equation}\label{eqn:znuniqueness}
z'(t)=(A_n-BR^{-1}B^*P(t))z(t)-BR^{-1}B^*\hat r(t), \quad z(0)=0,
\end{equation}
and let $\hat r_n$ be  the unique strict solution to
$$
r'(t)+(A_n^*-P(t)BR^{-1}B^*)r(t)-\overline Q S\hat z(t)=0, \quad r(T)=-\overline Q_TS_T\hat z(T).
$$
We are going to prove that $\hat z_n \rightarrow  \hat z$ and $\hat r_n \rightarrow \hat r$  uniformly on $[0,T]$ as $n \to \infty$. We prove that  for $\hat z_n$;  the same arguments can be applied to $\hat r_n$.

Since $\hat z_n$  is a strict solution to \eqref{eqn:znuniqueness}, it is also a mild solution to the same equation; that is, for all $t\in[0,T]$,
\begin{equation}\label{eq:mildznhat}
\hat z_{n}(t)=e^{t A_{n}}z_{0}-\int_{0}^{t} e^{(t-s)A_{n}} BR^{-1}B^*P(s)\hat z_{n}(s) ds- \int_0^{t}e^{(t-s)A_{n}}BR^{-1}B^{*}\hat r(s)ds.
\end{equation}
On the other hand, since  $\hat z$ is a mild solution to \eqref{eq:barz}, we also have
\begin{equation}\label{eq:mildbarz}
\hat z(t)=e^{t A}z_{0}-\int_{0}^{t} e^{(t-s)A} BR^{-1}B^*P(s)\hat z(s) ds- \int_0^{t}e^{(t-s)A}BR^{-1}B^{*}\hat r(s)ds.
\end{equation}
By \eqref{eq:mildznhat} and \eqref{eq:mildbarz}, we then get
\begin{align}\label{eq:diffbarznbarz}
\hat z_n(t)-\hat z(t) =&\ e^{t A_{n}}z_{0}-e^{t A}z_{0}-\int_{0}^{t} e^{(t-s)A_{n}} BR^{-1}B^*P(s)\hat z_{n}(s) ds\nonumber\\& +\int_{0}^{t} e^{(t-s)A} BR^{-1}B^*P(s)\hat z(s) ds+ \int_0^{t}\left(e^{(t-s)A}-e^{(t-s)A_{n}}\right)BR^{-1}B^{*}\hat r(s)ds.
\end{align}
Using the equality
\begin{eqnarray*}
&&-\int_{0}^{t} e^{(t-s)A_{n}} BR^{-1}B^*P(s)\hat z_{n}(s) ds+\int_{0}^{t} e^{(t-s)A} BR^{-1}B^*P(s)\hat z(s) ds\\ &=&\int_{0}^{t} e^{(t-s)A_{n}} BR^{-1}B^*P(s)(\hat z(s)-\hat z_{n}(s) )ds+\int_{0}^{t} \left(e^{(t-s)A}-e^{(t-s)A_n}\right) BR^{-1}B^*P(s)\hat z(s)ds
\end{eqnarray*}
into \eqref{eq:diffbarznbarz}, we get
\begin{align}\label{eq:diffbarznbarz2}\nonumber
|\hat z_n(t)-\hat z(t)|_H\leq & \left|e^{t A_{n}}z_{0}-e^{t A}z_{0}\right|_H+\int_{0}^{t}\left |\left(e^{(t-s)A}-e^{(t-s)A_{n}} \right)BR^{-1}B^*P(s)\hat z(s)\right |_Hds\\&+\int_{0}^{t} \left\|e^{(t-s)A_n} BR^{-1}B^*P(s)\right\|_H\left|\hat z(s)-\hat z_n(s) \right|_Hds\\&+ \int_0^{t}\left|\left(e^{(t-s)A}-e^{(t-s)A_{n}}\right)BR^{-1}B^{*}\hat r(s)\right|_Hds.\nonumber
\end{align}
Set
\begin{align*}
h_n(t):=& \ \left|e^{t A_{n}}z_{0}-e^{t A}z_{0}\right|_H+\int_{0}^{t} \left|\left(e^{(t-s)A}-e^{(t-s)A_{n}} \right)BR^{-1}B^*P(s)\hat z(s) \right|_Hds\\&+ \int_0^{t}\left|\left(e^{(t-s)A}-e^{(t-s)A_{n}}\right)BR^{-1}B^{*}\hat r(s)\right|_H\,ds,
\end{align*}
and
$$
g_n(t,s):=\left\|e^{(t-s)A_n} BR^{-1}B^*P(s)\right\|_{\mathcal{L}(H)}.
$$
Then, \eqref{eq:diffbarznbarz2} reads as
$$
|\hat z_n(t)-\hat z(t)| _H\leq h_n(t)+\int_{0}^{t} g_n(t,s)|\hat z(s)-\hat z_n(s) |_Hds.
$$
By Gronwall's Lemma, we therefore get
\begin{equation}\label{eqn:befgro}
|\hat z_n(t)-\hat z(t)|_H\leq h_n(t)+\int_0^th_n(s)g_n(t,s)e^{\int_s^tg_n(t,r)dr}ds.
\end{equation}
\item[\emph{Step 3.}]
We now want to take the limit as $n\to\infty$ in \eqref{eqn:befgro}.
First, by \eqref{semiconv}, we have
$$
|e^{tA_n}z_0- e^{tA}z_0|_H\rightarrow 0 \quad \mbox{ as } n \to \infty.
$$
Second,
noticing that, by \eqref{supP}, \eqref{sexiest}, and \eqref{eqn:barrbounded}, we have
$$
\left|\left(e^{(t-s)A}-e^{(t-s)A_{n}} \right)BR^{-1}B^*P(s)\hat z(s) \right|_H\leq C \quad \forall t \in [0,T],
$$
by \eqref{semiconv} and the dominated convergence Theorem we get
\begin{eqnarray*}
\int_{0}^{t} \left|\left(e^{(t-s)A}-e^{(t-s)A_{n}} \right)BR^{-1}B^*P(s)\hat z(s) \right|_Hds \rightarrow 0 \quad \mbox{ as } n \to \infty \ \  \forall t \in [0,T].
\end{eqnarray*}
Similarly,
$$
\int_{0}^{t} \left|\left(e^{(t-s)A}-e^{(t-s)A_{n}} \right)BR^{-1}B^*P(s)\hat r(s) \right|_Hds \rightarrow 0 \quad \mbox{ as } n \to \infty \ \  \forall t \in [0,T].
$$
%
%
Then, we have
$$
h_n(t)\to 0 \quad \mbox{ as } n \to \infty,\mbox{ for all } t \in [0,T].
$$
On the other hand,  by \eqref{supP} and \eqref{sexiest} we have
$$
g_n(t,s)\leq C \quad \forall \, 0\leq s\leq t\leq T.
$$
Then, by the dominated convergence theorem, we get
$$
|\hat z_n(t)-\hat z(t)|_H \rightarrow 0 \quad \mbox{ as } n \to \infty,  \ \forall t \in [0,T].
$$
\item[\emph{Step 4.}]
Consider now the sequence
$$
f_n: [0,T]\to \mathbb{R}, \ \ \ \ f_n(t):=\langle \hat z_n(t), \hat r_n(t)\rangle_{H} \quad \forall t \in [0,T].
$$
Then,
$$f_n(t)\rightarrow f(t):=\langle \hat z(t), \hat r(t)\rangle_{H}, \ \ \ \  \mbox{uniformly in  } t\in[0,T].$$
Moreover,
$$
f_n'(t)=-\langle BR^{-1}B^*\hat r(t),\hat r_n(t)\rangle_{H}+\langle \hat z_n(t), \overline QS\hat z(t)\rangle_{H} \quad\forall t \in [0,T].
$$
Setting
$$w(t):=-\langle BR^{-1}B^*\hat r(t),\hat r(t)\rangle_{H}+\langle \hat z(t), \overline QS\hat z(t)\rangle_{H}, \ \ \ \ \ t\in[0,T],$$
 we have
$$
f_n'(t)\rightarrow w(t) \mbox{ uniformly on } t\in [0,T].
$$
We may conclude that $f$ is differentiable and $f'=w$. This means that
\begin{equation}\label{man0}
\frac{d}{dt}\langle \hat z(t), \hat r(t)\rangle_{H}=-\langle R^{-1}B^*\hat r(t),B^{*}\hat r(t)\rangle_{H}+\langle \hat z(t), \overline QS\hat z(t)\rangle_{H} \quad \forall t \in [0,T].
\end{equation}
\blu{
\item[\emph{Step 5.}]Considering that $\hat z(0)=0$, and using Assumption \ref{assumptions}(vi)  and Assumption  (a) of the present theorem (on $\overline QS$), we get
\begin{equation}\label{pasta}
\langle \hat z(0), \hat r(0)\rangle_{H}=0, \quad \ \ \ \frac{d}{dt}\langle \hat z(t), \hat r(t)\rangle_{H}\leq 0.
\end{equation}
\blu{Moreover, recalling the terminal condition in \eqref{eqn:overliner} and using  again Assumption (a) of the present theorem (this time, on $\overline Q_{T}S_{T}$)}, we have
\begin{equation}\label{man02}
 \langle \hat z(T), \hat r(T)\rangle_{H}=\langle \hat z(T), -\overline Q_TS_T\hat z(T)\rangle_{H}\geq 0.
\end{equation}
Therefore,  combining \eqref{pasta} and \eqref{man02}, we obtain
\begin{equation}\label{zero}
\langle \hat z(t), \hat r(t)\rangle_{H}=0 \quad \forall t \in [0,T].
\end{equation}
From \eqref{pasta}-\eqref{man02}, it follows
$
\langle \hat z(T), -\overline Q_TS_T\hat z(T)\rangle_{H}=0.
$
Hence, by Assumption  (b) of the present theorem  (on $\overline Q_{T}S_{T}$), we get
\begin{equation}
\label{bas}
-\overline Q_TS_T \hat z(T)=0.
\end{equation}
Moreover,
using \eqref{zero} and using Assumption \ref{assumptions}(vi)   into \eqref{man0}, we get
$$
\langle \hat z(t), \overline QS\hat z(t)\rangle_{H}=0  \quad \forall t \in [0,T].
$$
So, again by Assumption (b) of the present theorem (this time, on $\overline QS$), we also get
\begin{equation}\label{bas2}
\overline QS\hat z(t)=0  \  \ \ \forall t\in[0,T].
\end{equation} Then, using \eqref{bas}-\eqref{bas2} into \eqref{eqn:overliner},  we get $\hat r(t)=0$ for all  $t\in[0,T]$. Finally, using the latter into  \eqref{eq:barz}, we also deduce $\hat z(t)=0$ for all $t\in[0,T]$, concluding the proof.}
\end{itemize}

\end{proof}

\section{Application: a Production Output Planning Problem with Delay}\label{sec:appl}
\subsection{The model}
We analyse a delayed version of a production output planning example introduced in \cite{CH}, Example A.

Consider $n$ firms $F_i$,  with $i=1,..., n$, supplying the same product to the market. Let us denote by $k_i$ be the production level of firm $F_i$ and suppose $k_i$ is subject to the following controlled stochastic dynamics:
\begin{equation}\label{eqn:firms}
\begin{cases}
dk_i(s)=\left[\alpha_i(s)+\int_{-d}^0b(\xi)\alpha_i(s+\xi)d\xi\right]ds+\sigma \,dW_i(s), \quad 0\leq s\leq T,\medskip\\
k_i(0)=k_i^0,  \ \ \quad \alpha_i(\xi)=\delta_{i}(\xi) \quad \forall \xi \in[-d,0].
\end{cases}
\end{equation}
Here $\alpha_i(s)$ is the control variable denoting the rate of investment/disinvestment in new capacity at time $s$. A fraction of this investment is immediately productive --- this is accounted by the term $\alpha_i(s)$ --- whereas another part takes time to become productive (so called time-to-build); overall, this second part is represented by the term $\int_{-d}^0b(\xi)\alpha_i(s+\xi)d\xi$, where $b\in L^2([-d,0];\mathbb{R}_+)$ is a kernel, and we refer to \cite{BDFG} for a detailed foundation of this last modeling feature. From the mathematical point of view, this term makes the problem a control problem with delay in the control variable, as can be appreciated also by the need of specifying an ``initial past'' for the control variable  by setting $\alpha_i(\xi)=\delta_{i}(\xi)$  for all $\xi \in[-r,0]$.  Finally, we add an idiosyncratic noise to the model represented by the term  $\sigma dW_i$ in the dynamics,  being the $W_i$'s  independent Brownian motions.

Setting $\overline k(t):=
 \frac{1}{n}\sum_{i=1}^nk_i(t)
$,
we assume that the price of the product is determined according to a linear inverse demand function:
\begin{equation}\label{eqn:price}
p(t)= \eta- \gamma \overline k(t),
\end{equation}
where $ \eta,  \gamma >0$ are given parameters.\footnote{Note that the overall production level $Q(t)=\sum_{i=1}^n z_i(t)$ is scaled by a factor $\frac{1}{n}$ in the price function, since we model the situation in which an increasing number of firms distributed over different areas join together to serve an increasing number of consumers (see \cite{CH}). The model \eqref{eqn:price} is a simplified form of a more general price model for many agents producing same goods proposed in \cite{Lamb}.}
The firm $F_i$ adjusts the production level $x_i$ looking at the current price of the product, considering that increasing price calls for more supplies of the product to consumers and viceversa.
The aim of the firm is to find a production level which is close to the price that the current market provides, i.e. $k_i(t) \approx \beta p(t)$, where
$\beta>0$ is a constant. Precisely,  the finite horizon cost of firm $F_i$ is
\begin{align*}
J_i\left(k_i^{0},\alpha_i;\overline k(\cdot)\right)&=\frac{1}{2} \,\,\mathbb{E}\left[\int_0^T \left[\left(k_i(t)-\beta p(t)\right)^2+r\alpha_i(t)^2\right]dt+ \left(k_i(T)-\beta p(T)\right)^{2}\right],
\\
&=\frac{1}{2}\,\,
\mathbb{E}\left[\int_0^T \left[\left(k_i(t)-\beta  \eta+\beta  \gamma \overline k(t)\right)^2+r\alpha_i(t)^2\right]dt+\left(k_i(T)-\beta  \eta+\beta  \gamma \overline k(T)\right)^{2}\right],
\end{align*}
where $r>0$.
We aim at studying the case in which there is a big number of firms and solve the problem by investigating the associated mean field game system. As usual in mean field game, one considers the formal limit of the Nash equilibrium as $n\to\infty$ and considers the optimization problem for the representative agent in the mean field.
Denoting by $\nu(t)$ the distribution of the population at time $t$ taken as given, the optimization problem of the representative agent is to minimize
$$
J\left(k^{0},\alpha;\nu\right)=\frac{1}{2}\,\,\mathbb{E}\left[\int_0^T \left[\left(k(t)-\beta  \eta+\beta  \gamma \int_{\R} \xi \nu(\textcolor{blu}{t,}d\xi)\right)^{2}+r\alpha(t)^{2}\right]dt+\left(k(T)-\beta  \eta+\beta  \gamma \int_{\R} \xi \nu(\textcolor{blu}{T,}d\xi)\right)^{2}\right],
$$
where
\begin{equation}\label{eqn:firmstypical}
\begin{cases}
dk(s)=\left[\alpha(s)+\int_{-d}^0b(\xi)\alpha(s+\xi)d\xi\right]ds+\sigma dW(s), \quad 0\leq s\leq T,\medskip\\
k(0)=k^0,  \quad \alpha(\xi)=\delta(\xi) \quad \forall \xi \in[-d,0],
\end{cases}
\end{equation}
where $W$ is a reference Brownian motion defined on a filtered probability space $(\Omega, \mathcal{F}, \mathbb{F}, \mathbb{P})$, satisfying the usual conditions and $\alpha$ belongs to $\mathcal{U}:=L^2_{\mathbb{F}}([0,T];\mathbb{R})$, the space of square integrable processes adapted to $\mathbb{F}$.
%
For future references, note that
$$
 \left(k-\beta  \eta+\beta  \gamma \int_{\R} \xi \nu(d\xi)\right)^{2}+r\alpha^2=g_1\left(k,\alpha, \nu\right)+g_2\left(k,\nu\right),\\
 $$
where
$$
g_1\left(k,\alpha,\nu\right)=\left(k+\beta \gamma\int_{\R}\xi \nu(d\xi)\right)^2+r\alpha^2,
\quad g_2\left(k,\nu\right)=(\beta \eta)^2-2\beta  \eta k-2 \eta\beta^2  \gamma \int_{\R}\xi \nu(d\xi).
$$

\subsection{Reformulation in Hilbert space}\label{subsec:ref}
We preceed to reformulate the problem in a suitable Hilbert space. We follow \cite{GM}, where the same reformulation has been carried out for a more general setting.
Let us consider the space
$
H=\mathbb{R}\times L^2_{d},$
where $L^2_{-d}:=L^2([-d,0]; \mathbb{R})$. We denote the generic element of $H$ by $x=(x_0,x_1(\cdot))$, where $x_0$ and $x_1(\cdot)$ denote, respectively, the $\mathbb{R}-$valued and the $L^2_{-d}-$ valued components. $H$ is a Hilbert space when
endowed with  inner product and norm
$$
\langle x,y\rangle_{H}=x_0y_0+\int_{-d}^0x_1(\xi)y_1(\xi)d\xi, \ \ \ \
|x|^{2}_{H}=|x_0|^2+\int_{-d}^0|x_1(\xi)|^2d\xi.
$$
Consider the linear closed unbounded operator
$$
A\, : \, D(A)\subset H \to H, \ \ \ \  (x_0, x_1(\xi))\mapsto \left(x_1(0), -\frac{dx_1(\xi)}{d\xi}\right),
$$
with domain
$$
D(A)=\left\{(x_0,x_1(\cdot)) \in H\, : \ \ x_1(\cdot) \in W^{1,2}([-d,0]; \mathbb{R}), \ x_1(-d)=0\right\}.
$$
The operator $A$ is the adjoint of the linear closed unbounded operator
$$
A^*\, : \, D(A^*) \subset H \textcolor{blu}{\rightarrow} H, \ \ \ \   \, (x_0, x_1(\cdot))\mapsto \left(0,\frac{dx_1(\xi)}{d\xi}\right),
$$
with domain
$$
D(A^*)=\left\{(x_0, x_1(\cdot) \in \mathbb{R}\times W^{1,2}([-d,0];\mathbb{R})\,: \, x_0=x_1(0)\right\}.
$$
It is well known that $A^*$ is the infinitesimal generator of a strongly continuous semigroup (see, e.g., Chojnowska-Michalik \cite{CM} or Da Prato and Zabczyk \cite{DaPraZa}), so it is  $A$.
We define now the bounded linear control operator
$$
B \, : \, \R \textcolor{blu}{\to}H, \ \ \ \  \, \alpha \mapsto  (1, b(\cdot))\alpha.
$$
In \cite[Prop.\,2]{GM}, it is proved that \eqref{eqn:firmstypical} is equivalent to the following abstract SDE in the  Hilbert space $H$
\begin{equation}\label{eqab}
\begin{cases}
dX(t)=(AX(s)+B\alpha(s))ds + G dW(s),\\
X(0)=x,
\end{cases}
\end{equation}
where $x$ is a suitable transformation of the initial data $(k^{0},\delta(\cdot))$ and
$$
G \, : \, \mathbb{R}\mapsto H, \ \ \ \  \, w\mapsto (\sigma w, 0).
$$
Given $m:[0,T]\to \mathcal{P}(H)$ measurable,
the objective functional
of the representative agent is then rewritten as
$$ J(x, \alpha;m)=\mathbb{E}\left[\int_0^T  \hat f(X(t),\alpha(t), m(t))dt +  \hat h(X(T),m(T))\right],
$$
 where $X(\cdot)$ evolves according to \eqref{eqab} and
$$
 \hat f(x,  \alpha, \mu)=\frac{1}{2}\left(g_1(x_0,\alpha,\mu^{0})+g_2(x_0,\mu^{0})\right),
\ \ \ \ \  \hat h(x, \mu)=\frac{1}{2}\left(g_1(x_0,0,\mu^{0})+g_2(x_0,\mu^{0})\right),
$$
where $\mu^{0}\in\mathcal{P}(\R)$ is the marginal of $m\in\mathcal{P}(H)$ on $\R$, that is
$$
\mu^{0} (A)=\mu(A\times L^{2}), \ \ \ A\in\mathcal{B}(\R).
$$
The terms involving $g_{1}$ falls into our setting with the following specifications
$$
U=\R \ \ \ \mbox{and} \ \ \  R=2r,
$$
$$
 S,S_T\in\mathcal{L}(H), \quad S x= S_Tx=-\beta  \gamma x,
$$
$$
Q,Q_T\in\mathcal{L}(H), \quad Q =Q_T =0,
$$
$$
\overline Q, \overline Q_T\in\mathcal{L}(H), \quad \overline Q x=\overline Q_Tx=(x_0,0).
$$
The term involving $g_{2}$, despite the constant $(\beta \eta)^2$  is the linear term of the form
$$
x_{0}+ \beta  \gamma \int_{\R}\xi \nu(d\xi)=\langle x,\hat n\rangle_{H}+\beta\gamma \int_{H} \langle x, \hat n\rangle_{H}\mu(dx),
$$
where $\hat n:=(1,0)\in H$.
This term can be inserted in our analysis as well, at the price of small changes. Indeed,
following the same computations as in Section \ref{sec:notset}, we find the same Riccati equation as \eqref{Riccati:P} and the same equation for $z(\cdot)$ as \eqref{eq:z}. The only differences are in the equation for $r(\cdot) $ which now reads as
\begin{equation}\label{eq:r2}
 \begin{cases}
\displaystyle{r'(t)+ (A^*-P(t)BR^{-1}B^*)r(t)-\overline Q Sz(t)-(\beta  \eta,0)=0},\\\\
r(T)=-\overline Q_TS_Tz(T)-(\beta  \eta,0),
\end{cases}
\end{equation}
and in the explicit expression for $s$ in terms of $P,z,r$.

Then, we have the following result.
\begin{theorem}
There exists a unique LQM mild solution to the MFG above.
\end{theorem}
\begin{proof}
One can prove existence exactly as in Theorem \ref{Th:existence}.  Uniqueness follows as in the proof of  Theorem \ref{Th:uniqueness}, taking into account also  Remark \ref{rem:uniqueness}; indeed,  by definition of $\overline Q, \overline Q_T,S, S_T$, assumptions  (ii)-(iii) of that theorem are satisfied.
\end{proof}

\appendix\section{}\label{sec:app}
\subsection{A useful lemma}
\begin{lemma}\label{lem:convT}
Let $\{x_n\}_n \subset H$ and $x \in H$ be such that $x_n \rightharpoonup x $  and let $\{F_n\}_n \subset \mathcal{L}(H)$ be such that $F_n^{*}\to F$ pointwise.
Then
$$
F_nx_n \rightharpoonup Fx.
$$
\end{lemma}
\begin{proof}
Let $y \in H$. We write
$$\langle F_nx_n- Fx, y\rangle_H =\langle F_nx_n-Fx_n, y\rangle_H +
\langle Fx_n-Fx, y\rangle_H
$$
Since $x_n \rightharpoonup x$, we have
$$
\langle Fx_n-Fx, y\rangle_H =\langle x_n-x, F^*y\rangle_H \rightarrow 0.
$$
On the other hand,
$$
|\langle F_nx_n-Fx_n, y\rangle_H|=|\langle x_n, (F_n^*-F^*)y\rangle_H|\leq |x_n|_H\,|(F_n^*-F^*)y|_H\leq  \left(\sup_{n} |x_n|_H\right)\, |(F_n^*-F^*)y|_H\rightarrow 0,
$$
where we used that
$$
\sup_{n} |x_n|_H<\infty,$$
since  $x_n \rightharpoonup x$. The claim follows.
\end{proof}

\subsection{On Assumption \ref{ass:etan}}
The propositions below are concerned with  the validity of Assumption \ref{ass:etan}.
In order to study the backward Riccati equation
\eqref{eqn:riccatieta}   we perform a time inversion and study the
following  forward Riccati equation

\begin{equation}\label{eqn:riccatifor}
\begin{cases}
 \eta_n'(t)=&(-P(T-t)BR^{-1}B^*+A_n^*)\eta_n(t)+\eta_n(t)(A_n-BR^{-1}B^*P(T-t)) \\&-\overline Q S-\eta_n(t)BR^{-1}B^*\eta_n(t),\\\\
 \eta_n(0)=&-\overline Q_T S_T,
 \end{cases}
\end{equation}
The notion of strict solutions to the previous equations is analogous to that of Assumption \ref{ass:etan}.
\begin{proposition}\label{prop:tsmall}
There exists $ \tau>0$  such that there exists a unique strict solution to \eqref{eqn:riccatifor}, in the sense of Assumption \ref{ass:etan}, in the interval $[0,\tau]$.
\end{proposition}
\begin{remark}\label{rem:tsmall}
Note that  Proposition \ref{prop:tsmall} implies that, if $T>0$ is small enough --- precisely, smaller than the time $\tau$ of the same proposition --- then Assumption \ref{ass:etan} is satisfied.
\end{remark}
\begin{proof}
We will prove the existence of a solution in the ball
\begin{equation}\label{Brat}
B_{r,\tau}=\Big\{g \in C_s([0,\tau];\Sigma(H)): \, \|g(t)\|_{C_s([0,\tau]; \Sigma(H))}\leq r\Big\},
\end{equation}
for some $r,\tau>0$ to be fixed later and   not  depending on $n$.
In particular, once we have existence in $B_{r,\tau}$, since $r$ does not depend on $n$, it follows that (H2) holds.

Given $\tau>0$, consider the map
$$\Gamma_{n}:C_{s}([0,\tau] ; \Sigma(H))\to C_{s}([0,\tau] ; \Sigma(H)), \ \ \ f\mapsto \Gamma_{n} f,$$
defined, for $(t,x)\in[0,\tau]\times H$, by
\begin{align*}
&\Gamma_n(f)(t)x=-e^{tA_n^*}\overline Q_T S_Te^{tA_n}x-\int_0^te^{sA_n^*}\overline QSe^{sA_n}xds\\&-\int_0^t e^{(t-s)A_n^*}(P(\tau-s)BR^{-1}B^*f(s)+f(s)BR^{-1}B^*P(\tau-s)+f(s)BR^{-1}B^*f(s))e^{(t-s)A_n}x\,\,ds.
\end{align*}
A mild solution to   \eqref{eqn:riccatifor} is a fixed point of $\Gamma_{n}$.
Set
\begin{equation}\label{eqn:r}
r:=2M_T^2\|\overline Q_T S_T\|_{\mathcal{L}(H)}
\end{equation}
and choose $\tau>0$ such that the following two are true:
\begin{equation}\label{eqn:A}
M_T^2\Big\{\|\overline Q_T S_T\|_{\mathcal{L}(H)}+\tau \Big[\|\overline QS\|_{\mathcal{L}(H)}+2r\|P(\tau-t)\|_{C_s([0,\tau];\Sigma(H))}\|BR^{-1}B^*\|_{\mathcal{L}(H)}+ r^2\|BR^{-1}B^*\|_{\mathcal{L}(H)}\Big]\Big\}\leq r
\end{equation}
and
\begin{equation}\label{eqn:B}
\tau M_T^2\Big[2\|P(\tau-t)\|_{C_s([0,\tau];\Sigma(H))}\|BR^{-1}B^*\|_{\mathcal{L}(H)}+2r\|BR^{-1}B^*\|_{\mathcal{L}(H)}\Big]\leq \frac{1}{2}.
\end{equation}
Letting $f \in B_{r,\tau}$ and recalling \eqref{est:semi} and \eqref{eqn:A}, we have for all $t\in[0,\tau]$ and $x \in H$
\begin{multline*}
|\Gamma_n(f)(t)x|_H\\\leq M_T^2\Big\{\|\overline Q_T S_T\|_{\mathcal{L}(H)}+\tau \Big[ \|\overline QS\|_{\mathcal{L}(H)}+2r \|P(\tau-t)\|_{C_s([0,\tau];\Sigma(H))}\|BR^{-1}B^*\|_{\mathcal{L}(H)}+ r^2\|BR^{-1}B^*\|_{\mathcal{L}(H)}\Big]\Big\}|x|_H\leq r,
\end{multline*}
so that
$$
\Gamma_n(B_{r,\tau})\subseteq B_{r,\tau}.
$$
Moreover, for all $ t \in [0,\tau]$ and $x \in H$,
\begin{multline*}
\Gamma_n(f)(t)x-\Gamma_n(g)(t)x=\int_0^te^{(t-s)A^*_n}[P(\tau-s)BR^{-1}B^*(g(s)-f(s))+(g(s)-f(s))BR^{-1}B^*P(\tau-s)\\+f (s)BR^{-1}B^*(g(s)-f(s))+(g(s)-f(s))BR^{-1}B^*g(s)](s)e^{(t-s)A_n}xds
\end{multline*}
and then
\begin{eqnarray*}
& &\|\Gamma_n(f)-\Gamma_n(g)\|_{C_s([0,\tau]; \Sigma(H))}=\sup_{t \in [0,\tau]}\|\Gamma_n(f)(t)-\Gamma_n(g)(t)\|_{\mathcal{L}(H)}\\ &\leq&\tau M_T^2\Big[2\|P(\tau-t)\|_{C_s([0,\tau];\Sigma(H))}\|BR^{-1}B^*\|_{\mathcal{L}(H)}+2r\|BR^{-1}B^*\|_{\mathcal{L}(H)}\Big]\sup_{t \in [0,\tau]}\|f(t)-g(t)\|_{\mathcal{L}(H)} \\ &\leq& \frac{1}{2}\sup_{t \in [0,\tau]}\|f(t)-g(t)\|_{\mathcal{L}(H)}=:\frac{1}{2}\|f-g\|_{C_s[0,\tau];\Sigma(H))}
\end{eqnarray*}
where the last inequality follows from \eqref{eqn:B}.
Thus $\Gamma_n$ is a contraction in $B_{r,\tau}$ and by the Banach-Cacciopoli fixed point theorem, there exists a unique mild solution $f$ in $B_{r,\tau}$.
\medskip

Finally, since $A_{n},A_{n}^{*}\in\mathcal{L}(H)$, we clearly have that $f\in C^{1}_{s}([0,T];\Sigma(H))$ and that it is a strict solution to \eqref{eqn:riccatieta}, i.e.,  in the sense of Assumption \ref{ass:etan}.
\end{proof}

The next proposition deals with the possibility of prolonging the strict solutions from  local ones to  global ones. We need to set assumptions on $\overline Q S$, $\overline Q_T S_T$ to achieve the goal; \blu{notice that they correspond to parts of the assumptions required for the global uniqueness (cf. (a) in Theorem \ref{Th:uniqueness}).}
\begin{proposition}\label{prop:positive}
If
  $-\overline Q S, -\overline Q_T S_T\in \Sigma_{+}(H)$, then Assumption \ref{ass:etan} holds true.
  \end{proposition}
\begin{proof}
\begin{enumerate}
\item[\emph{Step 1.}]
Here we show that,  for each solution $f$ of \eqref{eqn:riccatifor} in $[0,T_0]$ with $T_0\leq T$,  we have
\begin{equation}\label{eqn:aprioriest}
0\leq f(t)\leq \beta_TI \quad \forall  t \in [0,T_0],
\end{equation}
with
\begin{equation}\label{betat}
\beta_T:=M_T^2\big(\|\overline Q_TS_T\|_{\mathcal{L}(H)}+T\|\overline Q S\|_{\mathcal{L}(H)}\big)e^{2TM_T^2\sup_{s \in [0,T_0]}\|P(s)\|_{\mathcal{L}(H)}\|BR^{-1}B^*\|_{\mathcal{L}(H)}}.
\end{equation}
\begin{enumerate}[(i)]
\item Here we prove the lower bound of \eqref{eqn:aprioriest}, i.e., that
\begin{equation}\label{eq:poseta}
f(t)\geq 0 \quad \forall t \in [0,T_0].
\end{equation}
Note that $f$ is the solution of the following  in $[0,T_0]$:
$$
f'(t)=L_n(t,f(t))^*f(t)+f(t)L_n(t,f(t))-\overline Q S, \quad f(0)=-\overline Q_T S_T,
$$
where $$L_n(t,\varphi(t))=A_n-BR^{-1}B^*P(T-t)-\frac{1}{2}BR^{-1}B^*\varphi(t).$$ Denote by $U^{f}_n(t,s)$, where $0 \leq s \leq t \leq \tau$, the evolution operator associated with $L_n(t, f(t))$. Then
$$
f(t)=-U^{f}_n(t,0)\overline Q_TS_TU_n^{f}(t,0)^{*}-\int_0^tU^{f}_n(t,s)\overline QS U^{f}_n(t,s)^{*}ds,
$$
Since $-\overline Q S, -\overline Q_TS_T \in \Sigma^+(H)$, we get \eqref{eq:poseta}.
\medskip
\item
Here we prove the upper bound of \eqref{eqn:aprioriest}, i.e., that
\begin{equation}\label{lb}
f(t)\leq \beta_TI \quad \forall  t \in [0,T_0],
\end{equation}
where $\beta_{T}$ is given in \eqref{betat}.
For $t \in [0,T_0]$ and $x \in H$, we have
\begin{eqnarray*}
\langle f(t)x,x\rangle_H&=&-\langle \overline Q_T S_T e^{tA_n}x, e^{tA_n}x\rangle_H-\int_0^t\langle \overline Q Se^{sA_n}x, e^{sA_n}x\rangle_H\,ds\\&-&\int_0^t\langle P(T_0-s)BR^{-1}B^*)f(s)e^{(t-s)A_n}x,e^{(t-s)A_n}x\rangle_Hds\\ &-&\int_0^t\langle f(s)BR^{-1}B^*P(T_0-s)e^{(t-s)A_n}x, e^{(t-s)A_n}x\rangle_H\,ds\\ &-&\int_0^t\langle R^{-1}B^*f(s)e^{(t-s)A_n}x, B^*f(s)e^{(t-s)A_n}x\rangle_H\,ds.
\end{eqnarray*}

Since $R$ is nonnegative, for all $t \in [0,T_0]$ and $x \in H$, we have
$$
\int_0^t\langle R^{-1}B^*f(s)e^{(t-s)A_n}x, \,\,B^*f(s)e^{(t-s)A_n}x\rangle_Hds\geq 0.
$$
Therefore, for all $t \in [0,T_0]$ and $x \in H$ we have
\begin{eqnarray*}
\langle f(t)x,x\rangle_{H}&\leq&-\langle \overline Q_T S_T e^{tA_n}x, e^{tA_n}x\rangle_H-\int_0^t\langle \overline Q Se^{sA_n}x, e^{sA_n}x\rangle_H\,ds\\&-&\int_0^t\langle P(T_0-s)BR^{-1}B^*f(s)e^{(t-s)A_n}x,e^{(t-s)A_n}x\rangle_Hds\\ &-& \int_0^t\langle f(s)BR^{-1}B^*P(T_0-s)e^{(t-s)A_n}x, e^{(t-s)A_n}x\rangle_H\,ds,
\end{eqnarray*}
which implies
\begin{eqnarray*}
|\langle f(t)x,x\rangle_H|&\leq& \|\overline Q_TS_T \|_{\mathcal{L}(H)}M_T^2|x|_{H}^2+TM_T^2\|\overline Q S\|_{\mathcal{L}(H)}|x|_{H}^2\\&+&2M_T^2\sup_{t \in [0,T_0]}\|P(T_0-t)\|_{\mathcal{L}(H)}\|BR^{-1}B^*\|_{\mathcal{L}(H)}\int_0^t\|f(s)\|_{\mathcal{L}(H)}|x|_{H}^2ds.
\end{eqnarray*}
Then, by the characterization of the norm of a self-adjoint operator we have for all $t \in [0,T_0]$
\begin{multline*}
\|f(t)\|_{\mathcal{L}(H)}\leq M_T^2(\|\overline Q_TS_T \|_{\mathcal{L}(H)}+T\|\overline Q S\|_{\mathcal{L}(H)})\\+2M_T^2\sup_{t \in [0,T_0]}\|P(T_0-t)\|_{\mathcal{L}(H)}\|BR^{-1}B^*\|_{\mathcal{L}(H)}\int_0^t\|f(s)\|_{\mathcal{L}(H)}ds
\end{multline*}
and by the Gronwall's Lemma we have
$$
\|f(t)\|_{\mathcal{L}(H)}\leq M_T^2(\|\overline Q_TS_T\|_{\mathcal{L}(H)}+T\|\overline Q S\|_{\mathcal{L}(H)})e^{2TM_T^2\sup_{t \in [0,T_0]}\|P(T_0-t)\|_{\mathcal{L}(H)}\|BR^{-1}B^*\|_{\mathcal{L}(H)}}, \quad \forall t \in [0,T_0],
$$
from which we conclude \eqref{lb}.
\end{enumerate}

\item[\emph{Step 2.}]
Here we prove the existence of a strict solution to \eqref{eqn:riccatifor} in the whole interval $[0,T]$. First of all, by Proposition \ref{prop:tsmall}, we may construct a
 (unique) solution $f$ of \eqref{eqn:riccatieta} in $B_{r, \tau}$ given in \eqref{Brat} with $r$ and $\tau$ satisfying \eqref{eqn:r},  \eqref{eqn:A}, and \eqref{eqn:B}. We proceed by a second contraction on the ball
$$
B_{r_{1},\tau_{1}}=\Big\{g \in C_s([\tau,{\tau+\tau_{1}}];\Sigma(H)): \, \|g(t)\|_{C_s([\tau,\tau+\tau_{1}]; \Sigma(H))}\leq r_{1}\Big\},
$$
where $r_1$ and  $\tau_1$ have to be chosen appropriately. The initial datum at $t=\tau$ is $f(\tau)$ and we know that
$$
\|f(\tau)\|_{\mathcal{L}(H)}\leq \beta_T.
$$
Following the arguments in the proof of Proposition \ref{prop:tsmall}, we choose
$$
r_1:=2M_T^2\beta_{T}
$$

\begin{align*}
&\frac{r_1}{2}+\tau_1M_T^2\|\overline QS\|_{\mathcal{L}(H)}\\&+2r_1\tau_1M_T^2\|P(\tau-t)\|_{C_s([0,\tau+\tau_1];\Sigma(H))}\|BR^{-1}B^*\|_{\mathcal{L}(H)}+\tau_1r_1^2M_T^2\|BR^{-1}B^*\|_{\mathcal{L}(H)}\leq r_1.
\end{align*}
$$
\tau_1M_T^2\bigg[2\|P(\tau-t)\|_{C_s([0,\tau+\tau_1];\Sigma(H))}\|BR^{-1}B^*\|_{\mathcal{L}(H)}+2r_1\|BR^{-1}B^*\|_{\mathcal{L}(H)}\bigg]\leq \frac{1}{2}.
$$

By these choices we obtain a unique solution $f_1(t)$ on $[\tau, \tau+\tau_1]$ such that $f_1(\tau)=f(\tau)$. Then we stick the two solutions and we obtain a solution, that by some abuse of notation we call again $f$, on $[0,\tau+\tau_1]$. This solution satisfies the a priori estimate \eqref{eqn:aprioriest}. Hence
$$
\|f(t)\|_{\mathcal{L}(H)}\leq \beta_T \quad \forall t \in [0,\tau+\tau_1].
$$
This implies that we can iterate the contraction procedure on the ball with radius $r_1$ and interval $[\tau+\tau_1, \tau+2\tau_1]$ with the same choice of $r_1, \tau_1$. Then in a finite number $k$ of steps  we reach $T$ when $\tau+k\tau_1\geq T$. By \emph{Step 1} this solution satisfies (H2) and the proof is complete.
\end{enumerate}
\end{proof}

\section*{Acknowledgments}
\footnotesize{
Salvatore Federico, Daria Ghilli, Fausto Gozzi have been supported by the PRIN project ``The
Time-Space Evolution of Economic Activities: Mathematical Models and Empirical Applications".

Salvatore Federico and Daria Ghilli have been supported by the INdAM-GNAMPA project ``Modelli Matematici  per i Processi Decisionali riguardanti la Transizione Energetica".}


\begin{thebibliography}{100}
\bibitem{BDFG} M. Bambi, C. Di Girolami, S. Federico, F.Gozzi,  Generically distributed investments on flexible projects and endogenous growth, Economic Theory,  63, pp. 521-558, 2017.

\bibitem{BDPbook}
V. Barbu, G. Da Prato,
Hamilton-Jacobi Equations in Hilbert Spaces, Pitman, London, 1983.
\bibitem{Bardi12} M. Bardi, Explicit solutions of some linear-quadratic mean field games, Networks and heterogeneous media, vol. 7, no. 2,  2012.
\bibitem{BaPri} M. Bardi, F. S. Priuli, Linear-Quadratic $N$-person and Mean-Field Games with Ergodic Cost, SIAM Journal on Control and Optimization, vol. 52, iss. 5, 2014.
\blu{\bibitem{Boga} V. Bogachev, Measure theory, Springer, 2007.}
\bibitem{DPB} A. Benoussan, G. Da Prato, M. C. Delfour, S. K. Mitter, Representation and Control of Infinite Dimensional Systems, Systems and Control: Foundations and Applications, Birkh\"auser, 2006.
\bibitem{BeFrYa} A. Bensoussan, J. Frehse, P. Yam, Mean Field Games and Mean Field type Control, Springer, 2013.
\bibitem{BeSuYa} A. Bensoussan, K. C. J. Sung, S. C. P. Yam, S. P. Yung, Linear-Quadratic Mean Field Games. J Optim Theory Appl, vol. 169, 496-529, 2016.
\bibitem{BoDaPraRo3} V. I. Bogachev, G. Da Prato, M. R{\"o}ckner, Uniqueness for solutions of Fokker-Planck equations on infinite dimensional spaces,
Communications in Partial Differential Equations, vol. 36, 925-939, 2011.

\bibitem{BoDaPraRo5} Bogachev, V.I., Da Prato, G., R{\"o}ckner, M.: Existence and uniqueness of solutions for Fokker-Planck equations on Hilbert spaces. J. Evol. Equ., vol. 10, 487-509, 2010.
\bibitem{BoDaPraRoSh}  V. I. Bogachev, G. Da Prato, M. R{\"o}ckner, S. V.  Shaposhnikov, Nonlinear evolution equations for measures on infinite dimensional spaces, 2009, https://api.semanticscholar.org/CorpusID:208305441.
\bibitem{BoDaPraRoSh2}V. I. Bogachev, G. Da Prato, M. R{\"o}ckner, S. V.  Shaposhnikov,An analytic approach to infinite-dimensional continuity
and Fokker-Planck-Kolmogorov equations, Ann. Sc. Norm. Super. Pisa Cl. Sci., vol. 14, iss. 3, 983-1023, 2015.
\bibitem{BR} H. Brezis, Functional Analysis, Sobolev Spaces and Partial Differential Equations,Springer Science \& Business Media, 2010.
\bibitem{CH} P. E. Caines, M. Huang,  R. P. Malhamé, Large-Population Cost-Coupled LQG Problems With
Nonuniform Agents: Individual-Mass Behavior and
Decentralized $\eps$-Nash Equilibria,  IEEE Transactions on Automatic Control, vol. 52, no. 9, 2007.
\bibitem{HCM} P. E. Caines, M. Huang,  R. P. Malham\'e, Large population stochastic dynamic games: Closed-loop McKean-Vlasov systems and the Nash certainty equivalence principle,
Communications in Information and Systems, vol. 6, no. 3, 2006.
\bibitem{CDLLbook} P. Cardaliaguet, F. Delarue, J.-M. Lasry, P.-L. Lions, The master equation and the convergence problem in mean field games,  Annals of Mathematics Studied, no. 201, 2019.
\bibitem{CarmonaDelarueBook} R. A. Carmona, F. Delarue, Probabilistic Theory of Mean Field Games with Application I: Mean Field FBSDEs, Control, and Games,  Probability Theory and Stochastic Modelling (PTSM, vol. 83), 2018.
\bibitem{CarmonaDelarue2} R. Carmona, F. Delarue, Probabilistic Theory of Mean Field Games with Applications II:
Mean Field Games with Common Noise and Master Equations,  Probability Theory and Stochastic Modelling (PTSM, vol. 84), 2018.
\bibitem{CaFoSu} R. Carmona, J.P. Fouque, L.H. Sun, Mean field games and systemic risk, Communications in Mathematical Sciences, vo.l. 13,no. 4, 911-933, 2015.
\bibitem{CM} A. Chojnowska-Michalik,  Representation theorem for general stochastic delay equations,
Bull. Acad. Polon. Sci. Ser. Sci. Math. Astronom. Phys., vol. 26, iss. 7, 635-642, 1978.
\bibitem{CP74} R. F. Curtain, A. J. Pritchard, The infinite dimensional Riccati equation,  Journal of Mathematical Analysis and Applications, vol.  47, iss. 1, 43-57, 1974.
\bibitem{CP78} R. F. Curtain, A. J. Pritchard, Infinite Dimensional Linear Systems Theory, Springer Lecture Notes in Control and Informations, Science, Springer, New York, 1978.
\bibitem{DaPra} G. Da Prato, Fokker-Planck equations in Hilbert spaces, In: Eberle, A., Grothaus, M., Hoh, W., Kassmann, M., Stannat, W., Trutnau, G. (eds) Stochastic Partial Differential Equations and Related Fields. SPDERF 2016. Springer Proceedings in Mathematics $\&$ Statistics, vol 229. Springer, Cham.
\bibitem{DaPraFlaRo}G.  Da Prato, F.  Flandoli, M. R{\"o}ckner, Fokker-Planck Equations for SPDE with Non-trace-class Noise. Commun. Math. Stat., vol. 1, iss. 3, 281–304,2013.
\bibitem{DaPraZa} G. Da Prato, J. Zabczyk, Stochastic equations in infinite dimension, Cambridge University Press,  2014.
\bibitem{DiFe} J. Dianetti,  G. Ferrari, Nonzero-sum submodular monotone-follower games: existence and approximation of Nash equilibria, SIAM Journal on Control and Optimization, 58,  1257-1288, 2020.
\bibitem{DiFeFiNe} J. Dianetti, G. Ferrari, M. Fischer, M.  Nendel, Submodular Mean Field Games: Existence and Approximation of Solutions,The Annals of Applied Probability, vol. 31, iss. 6, 2538-2566, 2021.
\bibitem{DiFeFiNe2} J. Dianetti, G. Ferrari, M. Fischer, M. Nendel, A unifying framework for submodular mean field games, Mathematics of Operations Research, vol. 48, iss. 3, 2023.

\bibitem{EN}  K.-J. Engel, R. Nagel, One-parameter Semigroups  for linear Evolution Equation, Springer, 1999.

\bibitem{FGS} G. Fabbri, F. Gozzi, A. {\'{S}}wiech, Stochastic Optimal Control in Infinite
Dimension, Probability Theory and Stochastic Modelling,  vol. 82, 2017.
\blu{\bibitem{FeGoSw}
S. Federico, F. Gozzi, A. {\'{S}}wiech, On Mean Field Games in Infinite Dimension,
arXiv preprint arXiv:2411.14604, 2024.}
\bibitem{FZ} J.P. Fouque, Z. Zhang, Mean Field Game with Delay: A Toy Model, Risks, MDPI, vol. 6, iss.3,  1-17, 2018.
\bibitem{GM} F. Gozzi, C. Marinelli, Stochastic optimal control of delay equations arising in advertising models, SPDE and Applications, G. Da Prato \& L. Tubaro eds., VII, 133-148, 2006.
\bibitem{GM2} F. Gozzi,  C. Marinelli, and S. Savin, On Controlled Linear Diffusions with Delay in a Model of Optimal Advertising under Uncertainty with Memory Effects, J Optim Theory Appl, vol.142, 291–321, 2009.
\bibitem{GuaTe}  G. Guattieri, G. Tessitore, On the backward stochastic Riccati equation in infinite dimension, SIAM Journal on Control and Optimization, Vol. 44, Iss. 1, 2005.
\bibitem{Lamb} V. E. Lambson, Self-enforcing collusion in large dynamic markets,
J. Econ. Theory, vol. 34,  282–291, 1984.
\bibitem{LL1} J.-M. Lasry and P.-L. Lions, Jeux a champ moyen. I. Le cas stationnaire. C. R. Math. Acad. Sci. Paris, vol. 343, iss.9, 619-625,
2006.
\bibitem{LL2} J.-M. Lasry and P.-L. Lions, Jeux a champ moyen. II. Horizon fini et contr\^ole optimal.  C. R. Math. Acad. Sci. Paris,
vol. 343, iss. 10, 679-684, 2006.
\bibitem{LL3} J.-M. Lasry and P.-L. Lions, Mean field games, Jpn. J. Math., vol.2, iss.1, 229-260, 2007.
\bibitem{L} P.-L. Lions. Cours au coll\`ege de france. www.college-de-france.fr.
\bibitem{liu2024hilbert}
H.~Liu and D.~Firoozi.
\newblock {Hilbert Space-Valued {LQ} Mean Field Games: An Infinite-Dimensional
  Analysis}.
\newblock {\em arXiv: 2403.01012}, 2024.

\bibitem{LiYong} X. Li, J. Yong, Optimal Control theory for Infinite Dimensional Systems, Birkh\"auser Boston, 1995.
\bibitem{LiZh} T. Li, J.-F. Zhang,
Asymptotically Optimal Decentralized Control for
Large Population Stochastic Multiagent Systems, IEEE Transactions on automatic control, vol. 53, no. 7, 2008.
\blu{\bibitem{Parthasarathy67}
K.~R. Parthasarathy.
\newblock {\em Probability Measures on Metric Spaces},  {\em
  Probability and Mathematical Statistics, Vol. 3}.
\newblock Academic Press, New York, 1967.}

\end{thebibliography}
\end{document}